\documentclass[12pt]{article}
\usepackage[cm]{fullpage}

\usepackage{amsmath,amscd, float,times,rotating}
\usepackage{pb-diagram}

\usepackage{latexsym}
\usepackage{amsfonts}
\usepackage{amsmath}
\usepackage{amssymb}

\numberwithin{equation}{section}
\newcommand{\qed}{\hfill \ensuremath{\Box}}

\def\XXint#1#2#3{{\setbox0=\hbox{$#1{#2#3}{\int}$}
\vcenter{\hbox{$#2#3$}}\kern-.5\wd0}}

\newcommand{\tr}{\textnormal{tr}}
\newcommand{\ric}{\textnormal{Ric}}
\newcommand{\kod}{\textnormal{kod}}

\newcommand{\omegaphi}{\omega_{\varphi}}

\newcommand{\chiphibar}{\chi_{\bar{\varphi}}}
\newcommand{\omegapsi}{\omega_{\psi}}
\newcommand{\bvarphi}{\bar{\varphi}}

\newcommand{\dbar}{\overline{\partial}}

\newcommand{\ddt}[1]{\frac{\partial #1}{\partial t}}

\newcommand{\PP}{\mathcal{P}}

\newcommand{\OO}{\mathcal{O}}
\newcommand{\II}{\mathcal{I}}
\newcommand{\FF}{\mathcal{F}}
\newcommand{\AAA}{\mathcal{A}}

\newcommand{\Xcr}{ X_{can}^{\circ}}

\newcommand{\Xr} {X^{\circ}}

\newcommand{\hchi}{\hat{\chi}}

\newcommand{\hvarphi}{\hat{\varphi}}

\newcommand{\hF}{\hat{F}}

\newcommand{\hOmega}{\hat{\Omega}}

\newcommand{\ddbar}{\sqrt{-1}\partial\dbar}

\begin{document}
\newcounter{remark}
\newcounter{theor}
\setcounter{remark}{0} \setcounter{theor}{1}
\newtheorem{claim}{Claim}
\newtheorem{theorem}{Theorem}[section]
\newtheorem{proposition}{Proposition}[section]
\newtheorem{lemma}{Lemma}[section]
\newtheorem{definition}{Definition}[section]
\newtheorem{corollary}{Corollary}[section]
\newenvironment{proof}[1][Proof]{\begin{trivlist}
\item[\hskip \labelsep {\bfseries #1}]}{\end{trivlist}}
\newenvironment{remark}[1][Remark]{\addtocounter{remark}{1} \begin{trivlist}
\item[\hskip \labelsep {\bfseries #1
\thesection.\theremark}]}{\end{trivlist}}

\centerline{\bf\large CANONICAL  MEASURES AND K\"AHLER-RICCI FLOW
 \footnote{Research supported in
part by National Science Foundation grants DMS-0604805 and DMS-0302744.}}

\bigskip

\begin{center}{\large Jian Song$^{*}$   ~  and  ~ Gang Tian$^{\dagger}$}

\bigskip

\begin{abstract} We show that the K\"ahler-Ricci flow on an algebraic  manifold of positive Kodaira dimension and semi-ample canonical line bundle converges to a unique canonical metric on its canonical model. It is also shown that there exists a canonical measure of analytic Zariski decomposition on an algebraic manifold of positive Kodaira dimension. Such a canonical measure is unique and invariant under birational transformations under the assumption of the finite generation of canonical rings.

\end{abstract}

\bigskip

\end{center}

{\footnotesize \tableofcontents}

\bigskip
\bigskip

\section{Introduction}

It has been the subject of intensive study over the last few
decades to study the existence of K\"ahler-Einstein metrics on a compact K\"ahler
manifold, following Yau's solution to the Calabi conjecture (cf.
\cite{Ya2}, \cite{Au}, \cite{Ti2}, \cite{Ti3}). The Ricci flow (cf. \cite{Ha1, Ch}) provides a canonical deformation of K\"ahler metrics in K\"ahler geometry. Cao \cite{Ca} gave an alternative proof
of the existence of K\"ahler-Einstein metrics on a compact
K\"ahler manifold with trivial or negative first Chern class by the K\"ahler-Ricci flow.
However, most algebraic manifolds do not have a definite or
trivial first Chern class. It is a natural question to ask if
there exist any well-defined canonical metrics on these manifolds
or on varieties canonically associated to them.
Tsuji \cite{Ts1} applied the K\"ahler-Ricci flow and proved the existence of a canonical singular K\"ahler-Einstein metric on a minimal algebraic manifold of general type. It was the first attempt to relate the K\"ahler-Ricci flow and canonical metrics to the minimal model program. Since then, many interesting results have been achieved in this direction. The long time existence of the K\"ahler-Ricci flow on a minimal algebraic manifold with any initial K\"ahler metric is established in \cite{TiZha}. The regularity problem of the canonical singular K\"ahler-Einstein metrics on minimal algebraic manifolds of general type is intensively studied in \cite{Zh, EyGuZe1}.

In this paper, we propose a program of finding canonical measures on algebraic varieties of positive Kodaira dimension. Such a canonical measure can be considered as a birational invariant and it induces a canonical singular metric on the canonical model, generalizing the notion of K\"ahler-Einstein metrics.

Let $X$ be an $n$-dimensional compact K\"ahler manifold. A
K\"ahler metric can be given by its K\"ahler form $\omega$ on $X$.
In local coordinates $z_1, ..., z_n$, we can write $\omega$ as
$$\omega=\sqrt{-1}\sum_{i, j=1}^n
g_{i\bar{j}}dz_i\wedge dz_{\bar{j}},$$ where $\{g_{i\bar{j}}\}$ is
a positive definite hermitian matrix function.
Consider the normalized K\"ahler-Ricci flow
\begin{equation}
\label{krflow1}
\left\{
\begin{array}{rcl}
&& {\displaystyle \ddt{ \omega(t,\cdot)} = -
 \ric(\omega(t,\cdot))- \omega(t,\cdot)},\\
&&\\
&&\omega(0,\cdot)=\omega_0,
\end{array} \right.
\end{equation}
where $\omega(t,\cdot)$ is a family of K\"ahler metrics on $X$,
$ \ric(\omega(t,\cdot))$ denotes the Ricci curvature of
$\omega(t,\cdot)$ and $\omega_0$ is a given K\"ahler metric.

Let $X$ be a minimal algebraic manifold. If the canonical line bundle  $K_X$ of $X$ is ample and $\omega_0$
represents $[K_X]$, it is proved in \cite{Ca} that (\ref{krflow1}) has
a global solution $\omega(t,\cdot)$ for all $t\geq 0$ and
$\omega(t,\cdot)$ converges to a unique K\"ahler-Einstein metric of on $X$.
Tsuji showed in \cite{Ts1}  that (\ref{krflow1}) has a global solution
$\omega(t,\cdot)$ under
the assumption that the initial K\"ahler class $[\omega_0]
> [K_X]$.
This additional assumption was removed in \cite{TiZha}, moreover,
if $K_X$ is also big, $\omega(t,\cdot)$ converges to a singular
K\"ahler-Einstein metric with locally bounded K\"ahler potential
as $t$ tends to $\infty$ (see \cite{Ts1, TiZha}).

If $K_X$ is not big,  the Kodaira dimension of $X$ is smaller than its complex dimension. In particular, when $X$ is a minimal K\"ahler surface of  Kodaira dimension $1$,
it must be a minimal elliptic surface and does not
admit any K\"ahler-Einstein current in $-c_1(X)$, with
bounded local potential smooth outside a
subvariety. Hence, one does not expect that $\omega(t,\cdot)$
converges to a smooth K\"ahler-Einstein metric outside a subvariety of $X$ in general.

Let $f: X\rightarrow X_{can}$ be a minimal elliptic surface of
$\kod(X)=1$. Suppose all the singular fibres are given by $X_{s_1}=f^{-1}(s_1)$, ... ,
$X_{s_k}= f^{-1}(s_k)$ of multiplicity $m_i\in \mathbf{N}$, $i=1,...,
k$. In \cite{SoTi}, the authors proved that the K\"ahler-Ricci flow on $X$ converges  for any initial K\"ahler metric to a positive current
$\omega_{can}$ on the canonical model $X_{can}$ satisfying
\begin{equation}\label{elliptic-ke}
\ric(\omega_{can})=-\omega_{can}+\omega_{WP}+ 2\pi \sum_{i=1}^k\frac{m_i
-1}{m_i}[s_i],
\end{equation}
where $\omega_{WP}$ is the induced Weil-Petersson metric and
$[s_i]$ is the current of integration associated to the divisor
$s_i$ on $X_{can}$.  $\omega_{can}$ is called a generalized
K\"ahler-Einstein metric on $X_{can}$.
Moreover, the K\"ahler-Ricci flow is collapsing onto $X_{can}$ exponentially fast with uniformly bounded scalar curvature  away from the singular fibres.

The first result of this paper is to generalize the above convergence result on the K\"ahler-Ricci flow to algebraic manifolds of positive Kodaira dimension and semi-ample canonical bundle.

Let $X$ be an $n$-dimensional
algebraic manifold of Kodaira dimension $0<\kappa < n$. We assume that canonical line bundle $K_X$ is semi-ample, then the canonical ring $R(X, K_X)$ is finitely generated and the pluricanonical system induces an algebraic fibre space $f: X \rightarrow X_{can}$. Each nonsingular fibre of $f$ is a nonsingular Calabi-Yau manifold. We denote by $X_{can}^{\circ}$ the set of all nonsingular points $s\in X_{can}$ such that $f^{-1}(s)$ is a nonsingular fibre and let $X^{\circ}= f^{-1} (X_{can}^{\circ})$.
The $L^2$-metric
on the moduli space of nonsingular Calabi-Yau manifolds induces a semi-positive $(1,1)$-form $\omega_{WP}$ of Weil-Petersson type on
$X_{can}^{\circ}$.  We will study the K\"ahler-Ricci flow starting from any K\"ahler
metric and describe its limiting behavior as time goes to infinity.

\bigskip

\noindent{\bf Theorem A~} {\it
Let $X$ be a nonsingular algebraic variety with semi-ample
canonical line bundle $K_X$ and so $X$ admits an algebraic fibration  $f:
X\rightarrow X_{can}$ over its canonical model $X_{can}$.  Suppose $0< \dim X_{can}= \kappa < \dim X = n$. Then for any initial K\"ahler metric, the
K\"ahler-Ricci flow (\ref{krflow1}) has a global solution
$\omega(t,\cdot)$ for all time $t\in [0, \infty)$ satisfying:

\begin{enumerate}

\item $\omega(t,\cdot)$ converges to $f^*\omega_{can}\in -2\pi
c_1(X)$ as currents for a positive closed $(1,1)$-current $\omega_{can}$ on
$X_{can}$ with continuous local potential.

\item $\omega_{can}$ is smooth on $X_{can}^{\circ}$ and satisfies
the generalized K\"ahler-Einstein equation on $X_{can}^{\circ}$
\begin{equation}
\ric(\omega_{can})=-\omega_{can}+\omega_{WP}.
\end{equation}

\item for any compact subset $K\in X_{can}^{\circ}$, there is a constant
$C_K$ such that
\begin{equation}
||R(t, \cdot)||_{L^\infty(f^{-1}(K))}+ e^{(n- \kappa) t}
\sup_{s \in K} ||\omega^{n-\kappa}(t,\cdot) |_{X_s}||_{L^\infty(X_s)} \le C_K,
\end{equation}
where $X_s = f^{-1}(s)$.

\end{enumerate}
}

Therefore, the K\"ahler-Ricci collapses onto the canonical model with bounded scalar curvature away from the singular fibres and the volume of each nonsingular fibre tends to $0$ exponentially fast. In fact, the local potential of $\omega(t, \cdot)$ converges on $X^{\circ}$ locally in $C^0$-topology (cf. Proposition \ref{uniconv}). It should also converge locally in $C^{1,1}$-topology on $X^{\circ}$ as in the surface case (cf. \cite{SoTi}) and this will be studied in detail in a forthcoming paper.

Similar phenomena also appears in the real setting as a special type-III Ricci flow solution without the presence of singular fibres. It is discovered and intensively studied in \cite{Lo}.

The abundance conjecture in algebraic geometry predicts that the canonical line bundle is semi-ample if it is nef. If the abundance conjecture is true, then  Theorem A immediately implies that on all nonsingular minimal models of positive Kodaira dimension the K\"ahler-Ricci flow converges  to a unique canonical metric on their canonical model.

In general, the canonical line bundle of an algebraic manifold of positive Kodaira dimension is not necessarily semi-ample or even nef. The minimal model program in birational geometry deals with the classification of algebraic varieties and aims to choose a minimal model in each birational equivalence class. Tsuji claimed in \cite{Ts3} that there exists a singular K\"ahler-Einstein metric of analytic Zariski decomposition on algebraic manifolds of general type without assuming the finite generation of the canonical ring. Such a metric is constructed through a family of K\"ahler-Einstein metrics as the limits of a parabolic Monge-Amp\`ere equation of Dirichlet type. The approach is interesting but rather complicated.
The recent exciting development in the study of degenerate complex Monge-Amp\`ere equations (cf. \cite{Ko1, Zh, EyGuZe1}) enables the authors to give an  independent and correct proof.
\bigskip

\noindent{\bf Theorem B.1~} \label{keazd}  {\it Let $X$ be an algebraic manifold of general type. Then there exists a measure $\Omega_{KE}$ on $X$  such that

\begin{enumerate}

\item $(K_X, \Omega_{KE}^{-1})$ is an analytic Zariski
decomposition.

\item  Let $\omega_{KE} = \ddbar \log \Omega_{KE}$ be the closed positive $(1,1)$ current on $X$. Then there exists a non-empty Zariski open subset $U$  of
$X$ such that  $\ric(\omega_{KE})= \ddbar\log (\omega_{KE})^n$ is well-defined on $U$ and

$$\ric(\omega_{KE})=-\omega_{KE}.$$

\end{enumerate}

}

The proof of Theorem B.1 is given in Section 4.3. The existence of such a canonical K\"ahler-Einstein metric is also considered by Siu in \cite{Si1} as an alternative approach to attack the problem of the finite generation of canonical rings. A degenerate Monge-Amp\`ere equation of Dirichlet type is considered and the solution is expected to be unique. Indeed, if the canonical rings are finitely generated, such a solution coincide with the K\"ahler-Einstein metrics constructed in Theorem B.1. We hope that Theorem B.1 might help to gain more understanding of the finite generation of canonical rings from an analytic point of view.
Theorem B.1 can be generalized to algebraic manifolds of positive Kodaira dimension.

\bigskip

\noindent{\bf Theorem B.2~} {\it Let $X$ be an $n$-dimensional algebraic manifold of Kodaira dimension $0<\kappa<n$. There exists a measure $\Omega_{can}$  on $X$ bounded above such that
$\left(K_X, \Omega_{can}^{-1}\right)$ is an analytic Zariski decomposition.
Let $\Phi^{\dagger}: X^{\dagger} \rightarrow Y^{\dagger}$ be any Iitaka fibration of $X$ with $\pi^{\dagger}: X^{\dagger} \rightarrow X$ and $\Omega^{\dagger} = (\pi^{\dagger})^* \Omega_{can}$ . Then

\begin{enumerate}

\item $\left(K_{X^{\dagger}}, \left( \Omega^{\dagger} \right)^{-1} \right)$ is an analytic Zariski decomposition.

\item There exists a closed positive $(1,1)$-current $\omega^{\dagger}$   on $Y^{\dagger}$ such that $\left(\Phi^{\dagger}\right)^*\omega^{\dagger}= \ddbar \log \Omega^{\dagger}$ on a Zariski open set of $X^{\dagger}$. Furthermore, we have

    \begin{equation}
    (\omega^{\dagger})^{\kappa} = (\Phi^{\dagger})_* \Omega^{\dagger},
    \end{equation}
on $Y^{\dagger}$ and so  on a Zariski open set of $Y^{\dagger}$,
    \begin{equation}
    \ric( \omega^{\dagger} )=- \omega^{\dagger} +\omega_{WP}.
    \end{equation}

\end{enumerate}

}

\bigskip
The definition of $\omega_{WP}$ in Theorem B.2 is given in Section 4.4 and it is a generalization of the Weil-Petersson form induced from an algebraic  deformation space of algebraic manifolds of Kodaira dimension $0$.

In fact,  the hermitian metric $\Omega_{can}^{-1}$ (also $\Omega_{KE}^{-1}$) on $K_X$ constructed in the proof of Theorem B.2 (also Theorem B.1) has stronger properties than  the analytic Zariski decomposition. Let  $$\Psi_{X, \epsilon}=  \sum_{m=1}^{\infty} \sum_ {j=0}^{d_m} \epsilon_{m,j} |\sigma_{m,j}|^{\frac{2}{m}} , $$ where $\{ \sigma_{m,j}\}_{j=0}^{d_m}$ spans $H^0(X, mK_X)$ and $\{\epsilon_{m,j}>0 \}$ is a sequence such that $\Psi_{X, \epsilon}$ is convergent. $\Psi_{X, \epsilon}$ is a measure or a semi-positive $(n,n)$-current on $X$. Then from the construction of $\Omega_{can}$ in the proof of Theorem B.2 (also Theorem B.1)

\begin{equation} \frac{\Psi_{X, \epsilon}}{\Omega_{can}} < \infty.
\end{equation}

If the canonical ring $R(X, K_X)$ is finitely generated, one can replace $\Psi_{X, \epsilon}$ by

$$\Psi_{X}= \sum_{m=0}^M \sum_{j=0}^{d_m} |\sigma_{m,j}|^{\frac{2}{m}}$$
for some $M$ sufficiently large.

Recently, the finite generation of canonical rings on algebraic varieties of general type was proved independently by \cite{BiCaHaMc} and \cite{Si2}. By assuming the finite generation of canonical rings, Theorem B.1 and B.2 can be strengthened and the proof can be very much simplified. It turns out that the canonical measure in Theorem B.1 and B.2 is unique and invariant under birational transformations.

\bigskip

\noindent{\bf Theorem C.1~} {\it Let $X$ be an algebraic manifold of general type. If the canonical ring $R(X, K_X)$ is finitely generated, the K\"ahler-Einstein measure constructed in Theorem B.1 is continuous on $X$ and smooth on a Zariski open set of $X$. Furthermore, it is the pullback of the unique canonical K\"ahler-Einstein measure $\Omega_{KE}$ from $X_{can}$ satisfying \begin{equation} (\ddbar\log \Omega_{KE})^n = \Omega_{KE}.\end{equation}
}

\bigskip

The unique K\"ahler-Einstein metric with continuous local potential and the associated K\"ahler-Einstein measure on the canonical model $X_{can}$ is constructed in \cite{EyGuZe1}. The K\"ahler-Einstein measure in Theorem C.1 is invariant under birational transformations and so it can be considered as a birational invariant. Theorem C.1 can also be generalized to all algebraic manifolds of positive Kodaira dimension.

\bigskip\medskip

\noindent{\bf Theorem C.2~} {\it Let $X$ be an $n$-dimensional algebraic manifold of Kodaira dimension $0<\kappa<n$. If the canonical ring $R(X, K_X)$ is finitely generated, then there exists a unique canonical measure $\Omega_{can}$ on $X$ satisfying

\begin{enumerate}

\item  $0< \frac {\Psi_X} { \Omega_{can}}  < \infty$.

\item $\Omega_{can}$ is continuous on $X$ and  smooth on a Zariski open set of $X$.

\item Let  $\Phi: X \dasharrow X_{can}$ be the  pluricanonical map. Then there exists a unique closed positive $(1,1)$-current  $\omega_{can}$ with bounded local potential on $X_{can}$ such that  $\Phi^* \omega_{can}= \ddbar\log \Omega_{can}$ outside the base locus of the pluricanonical system. Furthermore, on $X_{can}$
$$ ( \omega_{can} )^{\kappa} = \Phi_* \Omega_{can},$$
and
$$\ric(\omega_{can})= -\omega_{can} + \bar{\omega}_{WP}.$$
\end{enumerate}

\noindent
In particular,  $\Omega_{can}$ is invariant under birational transformations.

}

\bigskip

$\bar{\omega}_{WP}$ is defined in Section 4.5 (cf. Definition \ref{push2}) and it coincides with $\omega_{WP}$ in Theorem B.2 on a Zariski open set of $X_{can}$ by choosing $Y^{\dagger}$ to be $X_{can}$. Theorem C.1 and Theorem 3.2 are proved in Section 4.5.

\bigskip


\section{Preliminaries}


\subsection{Kodaira dimension and semi-ample fibrations}

Let $X$ be an $n$-dimensional compact complex algebraic manifold and $ L \rightarrow X $ a holomorphic line bundle over $X$. Let $N( L) $ be the semi-group defined by $$N(L) = \{ m\in \mathbf{N} ~ | ~ H^0(X , L^m) \neq  0 \}.$$

Given any $m\in N(L)$,  the linear system $|L^m|= \mathbf{P} H^0(X, L^m)$ induces a rational map $\phi_m$
$$\Phi_m ~ : ~ X  \dashrightarrow \mathbf{CP}^{d_m}$$
by any basis $\{ \sigma_{m,0},~ \sigma_{m,1}, ~... ~, \sigma_{m, d_m} \}$ of  $H^0(X, L^m )$

$$\Phi_m (z) = \left[ \sigma_{m,0},~ \sigma_{m,1}, ~... ~, \sigma_{m, d_m} (z) \right], $$
where $d_m + 1 = \dim H^0(X, L^m) $.
Let $Y_m = \Phi_m (X) \subset \mathbf{CP}^{d_m}$ be the image of the closure of the graph of $\Phi_m$.

\begin{definition} The Iitaka dimension of $L$ is defined to be

$$\kappa(X, L) = \max _{ m\in N(L)} \{ \dim Y_m \}$$ if $N(L) \neq \phi$,  and
$\kappa(X, L) = -\infty $ if $N(L)= \phi$.

\end{definition}

\begin{definition}
Let $X$ be an algebraic manifold and $K_X$ the canonical line bundle over $X$. Then the Kodaira dimension $\kod(X)$ of $X$ is defined to be

$$\kod(X) = \kappa (X, K_X).$$

\end{definition}

The Kodaira dimension  is a birational invariant of an algebraic variety and the Kodaira dimension of a singular variety is  equal to that of its smooth model.

\begin{definition}

Let $L \rightarrow X$ be a holomorphic line bundle  over a compact algebraic manifold $X$. $L$ is called semi-ample if $L^m$ is globally generated for some $m>0$.

\end{definition}

For any $m\in \mathbf{N}$ such that $L^m$ is globally generated,  the linear system $|L^m|$ induces a holomorphic map $\Phi_m$
$$\Phi_m ~ : ~ X  \rightarrow \mathbf{CP}^{d_m}$$
by any basis  of  $H^0(X, L^m )$.
Let $Y_m = \Phi_m (X)$ and so $\Phi_m$ can be  considered as

$$ \Phi_m ~ : ~ X  \rightarrow Y_m .$$

The following theorem is well-known (cf. \cite{La, Ue}).

\begin{theorem}\label{safibration}

Let $L \rightarrow X$ be a semi-ample line bundle over an algebraic manifold $X$. Then there is an algebraic fibre space
$$\Phi_{\infty} : X \rightarrow Y $$
such that for any sufficiently large integer $ m $ with $L^m$ being globally generated,
$$Y_m = Y ~~~~ and ~~~~ \Phi_m = \Phi_{\infty}, $$
where $Y$ is a normal algebraic variety.
Furthermore, there exists an ample line bundle $A$ on $Y$ such that $ L^m = (\Phi_{\infty})^* A$.

\end{theorem}

If $L$ is semi-ample, the graded ring $R(X, L) = \oplus_{m\geq 0} H^0( X, L^m)$ is finitely generated and so $R(X , L)=\oplus_{m\geq 0} H^0(X, L^m)$ is the coordinate ring of $Y$.

\begin{definition} \label{iitaka}
Let $L \rightarrow X$ be a semi-ample line bundle over an algebraic manifold $X$. Then the algebraic fibre space
$\Phi_{\infty}: X \rightarrow Y$  as in Theorem \ref{safibration} is called the Iitaka fibration associated to $L$ and it is completely determined by the linear system $|L^m|$ for sufficiently large $m$.

In particular, if the canonical bundle $K_X$ is semi-ample, the algebraic fibre space  associated to $K_X$
$$f: X \rightarrow X_{can}$$
is called the Iitaka fibration of $X$, where $f=\Phi_{\infty}$ and $X_{can}$ is called the canonical model of $X$.
\end{definition}


\subsection{Iitaka fibrations}

In general, the canonical line bundle is not necessarily semi-ample, and the asymptotic behavior of the pluricanonical maps is characterized  by the following fundamental theorem on Kodaira dimensions due to Iitaka (cf. \cite{Ue}).

\begin{theorem} Let $X$ be an $n$-dimensional algebraic manifold of positive Kodaira dimension. Then for all sufficiently large $m\in N(K_X)$, the pluricanonical maps $\Phi_m: X \rightarrow Y_m$ are birationally equivalent to an algebraic fibre space
$$\Phi^{\dagger} : X^{\dagger} \rightarrow Y^{\dagger}$$ unique up to birational equivalence satisfying

\begin{enumerate}

\item  There exists a commutative diagram for sufficiently large $m\in N(K_X)$  \begin{equation}
\begin{diagram}
\node{X} \arrow{s,l,..}{{\small \Phi_m} }     \node{X^{\dagger}}  \arrow{w,t}{\pi^{\dagger}} \arrow{s,r}{{\small \Phi^{\dagger}}}\\
\node{Y_m}      \node{Y^{\dagger}} \arrow{w,b,..}{\mu_m}
\end{diagram}
\end{equation}
of rational maps with $\pi^{\dagger}$ and $\mu_m$ being birational.

\item $\dim Y^{\dagger} = \kod(X)$.

\item A very general fibre of $\Phi^{\dagger}$ has Kodaira dimension $0$.

\end{enumerate}

\end{theorem}


\subsection{Analytic Zariski decomposition}

Let $X$ be a compact complex manifold and $L$ be a holomorphic
line bundle on $X$ equipped with a smooth hermitian metric $h_0$.

A singular hermitian metric $h$ on $L$ is given by
$$h=h_0e^{-\varphi}$$ for some $\varphi\in L^1(M)$.

Let $\Theta_{h_0}$ be the curvature of $h_0$ defined by
$$\Theta_{h_0}=-\ddbar \log h_0.$$ Then the curvature $\Theta_h$  of $h$
as a current is defined by
$$\Theta_h=\Theta_{h_0}+\ddbar \varphi.$$

\begin{definition}
$L$ is called pseudoeffective if there exists a singular hermitian
metric $h$ on $L$ such that the curvature $\Theta_h$ is a closed
positive current. Let $$\PP_{h_0}(X)=\{ \varphi\in L^1(X)~|~
\Theta_{h_0}+\ddbar \varphi \geq 0 ~{\rm as ~ current} \}.$$

\end{definition}
 
\begin{definition}Let $\varphi\in \PP_{h_0}(X)$ and $h=h_0 e^{-\varphi}$. The
multiplier ideal sheaf $\II(h)\subset \OO_X(L)$ or $\II(\varphi)$
is defined by

$$\Gamma(U, \II(h))=\{ f\in \Gamma(U,
\OO_X(L))~|~|f|^2_{h_0}e^{-\varphi}\in L_{loc}^1(U)\}.$$

\end{definition}

The notion of analytic Zariski decomposition is an analytic analog of Zariski decomposition and it is introduced in \cite{Ts1} to study  a pseudoeffective line bundle.

\begin{definition}
 A singular hermitian metric $h$ on $L$ is an analytic Zariski
 decomposition  if

 \begin{enumerate}

\item $\Theta_h$ is a closed semi-positive current,

\item for every $m\geq 0$, the natural inclusion

$$H^0(X, \OO_X(mL)\otimes \II(h^m))\rightarrow H^0(X, \OO_X(mL))$$
is an isomorphism.

 \end{enumerate}

\end{definition}


\subsection{Complex Monge-Amp\`ere equations}

Let $X$ be an $n$-dimensional K\"ahler manifold and let $\omega$ be a smooth closed semi-positive $(1,1)$-form. $\omega$ is K\"ahler if it is positive and $\omega$ is called big if $[\omega]^n = \int_X \omega^n >0$.

\begin{definition} A quasi-plurisubharmonic function associated to $\omega$ is a function $\varphi: X \rightarrow [-\infty, \infty)$ such that for any smooth local potential $\psi$ of $\omega$, $\psi+\varphi$ is plurisubharmonic.  We denote by $PSH(X, \omega)$ the set of all quasi-plurisubharmonic functions associated to $\omega$ on $X$.

\end{definition}

The following comparison principle for quasi-plurisubhharmonic functions on compact K\"ahler manifolds is well-known.
\begin{theorem}\label{kol} Let $X$ be an $n$-dimensional K\"ahler manifold.   Suppose $\varphi,~\psi \in PSH(X, \omega)$ for a big smooth closed semi-positive $(1,1)$-form $\omega$. Then

$$\int_{\varphi<\psi} (\omega + \ddbar \psi)^n \leq \int_{\varphi < \varphi} (\omega + \ddbar \phi)^n.$$

\end{theorem}

In \cite{Ko1}, Kolodziej proved the fundamental theorem on the existence of continuous solutions to the Monge-Amp\`ere equation $ (\omega + \ddbar \varphi)^n = F \omega^n $, where $\omega$ is a K\"ahler form  and $F\in L^p(X, \omega^n)$ for some $p>1$. 
Its generalization was independently carried out in [Zh] and [EyGuZe1].
They proved  that there is a bounded solution when $\omega$ is
semi-positive and big. A detailed proof for the
continuity of the solution
was given in [DiZh] (also see [Zh] for an earlier and sketched proof).
These generalizations are summarized in the
following. 
\begin{theorem}\label{zhang} Let $X$ be an $n$-dimensional K\"ahler manifold and let $\omega$ be a big smooth closed semi-positive $(1,1)$-form. Then there exists a unique continuous solution to the following Monge-Amp\`ere equation

$$ (\omega + \ddbar \varphi)^n =  F\Omega, $$
where $\Omega>0$ is a smooth volume form on $X$, $F\in L^{p}(X, \Omega)$ for some $p>1$ and $\int_X F\Omega = \int_X \omega^n$.

\end{theorem}

Recently,  Demailly and Pali proved the following uniform estimate and we refer the readers to the general statement in \cite{DePa}. Such an $L^{\infty}$-estimate is also independently obtained in \cite{EyGuZe2}.

\begin{theorem}\label{D-P}
Let $X$ be an $n$-dimensional  K\"ahler manifold. Let $\Omega>0$ be a smooth volume form and $\omega$ be a smooth K\"ahler form such that $\omega \leq \omega_0$ for some smooth K\"ahler form $\omega_0$ on $X$.  Let
$\varphi\in PSH(X, \omega)\cap L^{\infty}(X)$ be a solution of the degenerate complex Monge-Amp\`ere equation
$
(\omega+ \ddbar \varphi)^n= F\Omega
$
with $F\in L^{p}(X)$ for some $p>1$.
Suppose \begin{enumerate}
\item $\displaystyle \int_X  \left( \frac{F}{ [\omega] ^n}\right) ^p ~\Omega \leq A$,

\item $\displaystyle \frac{\omega^n}{[\omega]^n \Omega}+ \displaystyle  \int_X \left(\frac{[\omega]^{n}\Omega}{\omega^n}\right)^{\epsilon} \Omega \leq B,$ for some $\epsilon>0$.

\end{enumerate}
Then $$\sup_X \varphi - \inf_X \varphi \leq C(\Omega, \omega_0, \epsilon, p,  A, B).$$

\end{theorem}

The estimate in Theorem \ref{D-P} assumes very weak dependence on the reference form $\omega$ and it is essential in deriving the $C^0$-estimate for the K\"ahler-Ricci flow on algebraic manifolds with positive Kodaira dimension and semi-ample canonical line bundle (cf. Section 5.2).


\section{Canonical metrics for semi-ample canonical bundle}


\subsection{Canonical metrics on canonical models}

Let $X$ be an $n$-dimensional complex algebraic manifold with semi-ample canonical line bundle $K_X$.  Fix $m \in N(K_X)$ sufficiently large and let $f = \Phi_m$, the Iitaka fibration of $X$ is then given by the following holomorphic map
$$ f: X \rightarrow X_{can} \subset \mathbf{CP}^{d_m}.$$
We assume that $0 < \kappa = \kod(X) < n$ and so $X$ is an algebraic fibre space over $X_{can}$.
Let
$$X_{can}^{\circ} = \{ y \in X_{can} ~ | ~ y ~ {\rm is ~ a ~ nonsingular  ~ and ~} X_y=f^{-1}(s) ~ {\rm is ~ nonsingular ~ fibre} \}$$
and $X^{\circ} = f^{-1} ( X_{can}^{\circ})$.
The following proposition is well-known.

\begin{proposition} We have $$K_X = \frac{1}{m} f^* \mathcal{O}(1).$$
For all $y \in X_{can}^{\circ}$, $K_{X_y}$ is numerically trivial and so  $c_1(X_y)=0$.

\end{proposition}

Thus $X$ can be considered as a holomorphic fibration of polarized Calabi-Yau manifolds over its canonical model $X_{can}$. Since $f: X \rightarrow X_{can} \in \mathbf{ CP }^{d_m}$ and $ - c_1(X) = \frac{1}{m} [ f^* \mathcal{O} (1) ]$, we can define
$$\chi = \frac{1}{m} \ddbar \log \sum_{j=0}^{d_m} \left| \sigma_{m,j} \right|^2  \in - 2\pi c_1(X)$$
as a multiple of the pulled back Fubini-Study metric on $\mathbf{CP}^{d_m}$ by a basis $\{ \sigma_{m,j} \}_{j=0}^{d_m} \subset H^0(X, K_X^m)$.

We can also consider $\chi$ as the restriction of the multiple  of the Fubini-Study metric on the normal variety $X_{can}$ and we identify $\chi$ and $f^*\chi$ for convenience.

Let $\Omega = \sum_{ j=0 } ^{ d_m} \left|\sigma_{m,j} \right|^2 $. Since $K_X^m$ is base point free, $\Omega$ is a smooth nondegenerate volume form on $X$ such that
$$\ddbar \log \Omega = \chi. $$

\begin{definition}

The pushforward $ f_* \Omega $  with respect to the holomorphic map $f: X \rightarrow X_{can}$ is defined as currents as the following. For any continuous function $\psi$ on $X_{can}$

$$\int_{X_{can}} \psi f_* \Omega = \int_X (f^* \psi )\Omega.$$

\end{definition}

\begin{lemma} On  $X_{can}^{\circ}$,

$$ f_* \Omega = \int_{X_y} \Omega. $$

\end{lemma}

\begin{definition}  We define a function  $F$ on $X_{can}$ by
\begin{equation} \label{vol-comp}
F =   \frac{ f_* \Omega }{\chi^{\kappa} } .
\end{equation}

\end{definition}

\begin{lemma}\label{sf}
Given  any K\"ahler class $[\omega]$ on $X$, there is a
smooth function $\psi$ on $X^{\circ}$ such that $\omega_{SF}:=\omega+ \ddbar   \psi$ is a closed semi-flat ($1,1$)-form in the
following sense: the restriction of $\omega_{SF}$ to each smooth
$X_y\subset X^{\circ}$ is a Ricci flat K\"ahler metric.
\end{lemma}

\begin{proof}
For each $y\in X^{\circ}_{can}$, let $\omega_y$ be the restriction of
$\omega$ to $X_y$ and $\partial_V$ and $\dbar_V$ be the
restriction of $\partial$ and $\dbar$ to $X_y$. Then by the Hodge
theory, there is a unique function $h_y$ on $X_y$ defined by
\begin{equation}
\left\{
\begin{array}{rcl}
&&\partial_V\dbar_V h_y=-\partial_V\dbar_V\log \omega_y^{n-\kappa}, \\
&&\\
&&\int_{X_y}e^{h_y}\omega_y^{n-\kappa}=\int_{X_y}\omega_y^{n-\kappa}.
\end{array} \right.
\end{equation}
By Yau's solution to the Calabi conjecture, there is a unique
$\psi_y$ solving the following Monge-Amp\`ere equation
\begin{equation}
\left\{
\begin{array}{rcl}
&&\frac{(\omega_y+\sqrt{-1}\partial_V\dbar_V\psi_y)^{n-\kappa}}
{\omega_y^{n-\kappa}}=e^{h_y}\\
&&\\
&&\int_{X_y}\psi_y\omega_y^{n-\kappa}=0.
\end{array} \right.
\end{equation}
Since $f$ is holomorphic, $\psi(z,s)=\psi_y(z)$ is well-defined as
a smooth function on $X^{\circ}$.

\qed
\end{proof}

For each  $y\in X_{can}^{\circ}$, there exists a holomorphic $(n-\kappa, 0)$ form $\eta$ on $X_y$ such that $\eta\wedge \overline{\eta}$ is a Calabi-Yau volume form and $\int_{X_y} \eta \wedge \overline{ \eta } = \int_{X_y} \left( \omega|_{X_y} \right)^{ n - \kappa}$.

\begin{definition} The closed $(n-\kappa, n-\kappa)$-current $\Theta$ on $X^{\circ}$ is defined to be

\begin{equation}
\Theta = \left( \omega_{SF} \right) ^{n-\kappa}.
\end{equation}
Let $\Theta_y = \Theta |_{X_y}$ for $y\in X_{can}^{\circ}$ be the restriction of $\Theta$ on a nonsingular fibre $X_y$. Then $\Theta_y$ is a smooth Calabi-Yau volume form with $$\int_{X_y} \Theta_y = [\omega]^{ n -\kappa } \cdot X_y = constant .$$

\end{definition}

We can always scale $[\omega]$ such that $[\omega]^{n -\kappa} \cdot X_y = 1 $ for $y \in \Xcr$.

\begin{lemma} On $X^{\circ}$, we have
\begin{equation}
\label{vol-comp2} f^*F =   \left( \frac{\Omega}{ \Theta \wedge \chi^{\kappa}} \right) .
\end{equation}
Furthermore,  $\Theta$ can be extended to $X$ as current such that
$ f^*F =   \left( \frac{\Omega}{ \Theta \wedge \chi^{\kappa}} \right) $ on $X$.

\end{lemma}

\begin{proof}

Let $\FF=  \left(  \frac{\Omega}{ \Theta \wedge \chi^{\kappa}} \right) $ be defined on $X^{\circ}$. We first show that $\FF$ is constant along each fibre $X_y$ for $y\in X_{can}^{\circ}$.

Since $\chi$ is the pullback from $X_{can}$,
we have
$$\sqrt{-1} \partial_V \dbar_V\log\Omega=
\sqrt{-1}\partial_V\dbar_V\log \Theta \wedge
\chi^{\kappa}=0$$
on each nonsingular fibre $X_y$. On the other hand, $\FF$ is smooth on each $X_y$ for $y \in X_{can}^{\circ} $, therefore  $\FF$ is constant
along  $X_y$ and $\FF$  can be considered as  the pullback of a
function from $X_{can}^{\circ}$.

Now we can show (\ref{vol-comp2}). Let $\psi$ be any smooth test function on $X_{can}^{\circ}$. Let $y_0$ be a fixed point in $\Xcr$.

\begin{eqnarray*}
\int_{\Xcr} \psi F \chi^{\kappa} &=&  \int_{\Xcr} \psi f_* \Omega=\int_{y \in \Xcr}  \psi \left( \int_{X_{y}} \Omega \right) =  \int_{\Xr} \psi \Omega\\
\end{eqnarray*}

On the other hand,

\begin{eqnarray*}
\int_{\Xcr} \psi \FF \chi^{\kappa} &=&  \int_{ y \in \Xcr} \psi \left( \frac{ \Omega }{ \Theta \wedge \chi^{\kappa} } \right) \left( \int_{X_y} \Theta_y  \right) \chi^{\kappa} \\
&=&  \int_{ y \in \Xcr} \int_{X_y}\psi \left( \frac{ \Omega }{ \Theta \wedge \chi^{\kappa} } \right)   \Theta_y  \wedge \chi^{\kappa} \\
& = &  \int_{ \Xr} \psi \Omega .
\end{eqnarray*}
Therefore $f^*F=\FF$.

\noindent Let $\psi'$ be any smooth test function on $X$. $\Theta$ can be extended as current to $X$ such that

$$\int_{X} \psi' f^*F \Theta \wedge \chi^{\kappa} = \int_X \psi' \Omega . $$

\qed
\end{proof}

\begin{proposition} \label{Fso} $F$ is smooth on $X_{can}^{\circ}$ and there exists $\epsilon >0 $ such that
$$F\in L^{1+\epsilon} (X_{can} ).$$

\end{proposition}

\begin{proof} Calculate

$$ \int_{X_{can}} F^{1+\epsilon} \chi^{\kappa} = \int_{ X  } \left( f^*F \right) ^{1+\epsilon} \chi^{\kappa} \wedge \Theta = \int_{X} \left( f^* F \right) ^{\epsilon} \Omega.$$
Also for any $y \in X_{can}^{\circ}$,  we can choose $z_0 \in X_y$ such that  on $X_y$ $\omega^{n -\kappa} (z_0) = \omega_{SF} ^{n -\kappa} (z_0)$ since $X_y$ is smooth and $\int_{X_y} \omega^{n-\kappa} = \int_{X_y} \omega_{SF}^{ n -\kappa}$. Then

\begin{eqnarray*}
| F (y) | &=&  \frac{\Omega} {\chi^{\kappa} \wedge \omega_{SF}^{ n -\kappa} } \\
& = &   \frac{\Omega } {\chi^{\kappa} \wedge \omega^{ n -\kappa} } ~\frac{\omega^{n -\kappa} } { \omega_{SF}^{ n -\kappa} } \left( z_0 \right) \\
&&\\
&=&  \frac{\Omega } {\chi^{\kappa} \wedge \omega^{ n -\kappa} } \left( z_0 \right) \\
&&\\
&\leq& \sup_{X_y} \frac{\Omega } {\chi^{\kappa} \wedge \omega^{ n -\kappa} }.
\end{eqnarray*}
Therefore $F$ is bounded by poles and $ \left( f^* F \right) ^{\epsilon} $ is integrable for sufficiently small $\epsilon>0$.

\qed
\end{proof}

\begin{proposition}
Let $\pi: Y \rightarrow X_{can}$ be a smooth model of $X_{can}$ by resolution of singularities of $X_{can}$. $ \pi^* F$ has at worst pole singularities on $Y$.

\end{proposition}

\begin{proof}

Let $D$ be a divisor on $X_{can}$ such that $X_{can}\setminus X_{can}^{\circ} \subset S_D$. Let $S_D$ be the defining section of $D$ and $h_D$ be the hermitian metric on the line bundle associated to $[D]$ such that $\pi^* \left( |S_D|^2_{h_D}\right) $ is a smooth function. For any continuous volume form  $\Omega'$ on $X$,
\begin{eqnarray*}
&& f_* \left( |S_D|^{2N}_{h_D} \Omega' \right) \\
&=& \int_{X_s}  |S_D|^{2N}_{h_D} \Omega' \\
&=& \left( \int_{X_s}  |S_D|^{2N}_{h_D}\left(  \frac{\Omega'}{\chi^{\kappa} \wedge \omega^{n -\kappa}} \right) \left( \omega \left|_{X_s} \right. \right) ^{n-\kappa} \right)\chi^{\kappa}
\end{eqnarray*}
for $s\in X_{can}^{\circ}$.
Since $|S_D|^{2N}_{h_D} \frac{\Omega'}{\chi^{\kappa} \wedge \omega^{n -\kappa}} < \infty$ for sufficiently large $N$,
there exists  a constant $C$ such that
$$0\leq  f_* \left( |S_D|^{2N}_{h_D} \Omega' \right) < \chi^{\kappa}.$$

Let $\FF_{N} = |S_D|^{2N}_{h_D} F$. Then on $X_{can}^{\circ}$

\begin{eqnarray*}
&& \ddbar \FF_{N} \\
&=& \ddbar \left( \int_{X_s} |S_D|^{2N}_{h_D} \frac{\Omega}{\chi^{\kappa}\wedge \omega^{n-\kappa} } \omega^{n-\kappa} \right) \\
&=& \int_{X_s}  \ddbar \left( |S_D|^{2N}_{h_D}\frac{\Omega}{\chi^{\kappa}\wedge \omega^{n-\kappa} } \right) \wedge \omega^{n-\kappa} \\
&\leq& C |S_D|^{2M}_{h_D}  \int_{X_s}   \omega^{n-\kappa + 1} \\
\end{eqnarray*}
for some sufficiently large $M$ by choosing $N$ sufficiently large.

Let $\eta $ be any semi-positive smooth $(\kappa-1, \kappa-1)$-form supported on $X_{can}^{\circ}$. Then
\begin{eqnarray*}
&& \int_{ X_{can} } \eta \wedge \ddbar \FF_N \\
%
%
&\leq & C |S_D|^{2M}_{h_D}  \int_X f^*\eta \wedge \omega^{n-\kappa+1} \\
&=& C |S_D|^{2M}_{h_D} \int_{X_{can}} \left(  \frac{ \int_{X_s}  \omega^{n-\kappa+1} \wedge \chi^{\kappa-1} } {\chi^{\kappa}}  \right) f^* \eta \wedge \chi \\
&\leq &  C' |S_D|^{2L}_{h_D} \int_{X_{can}} \eta \wedge \chi  \\
\end{eqnarray*}
if we chose $N$ sufficiently large.

Similar lower bound of $\ddbar \FF_N$ can be achieved and so for sufficiently large $N$,  on $X_{can}^{\circ}$

$$                   - C \chi \leq  \ddbar \FF_N \leq C \chi .$$

Let $\omega_Y$ be a K\"ahler metric on $Y$ and $\Delta_Y$ be the Laplace operator associated to $\omega_Y$. Then  for sufficiently large $N$

$$   \left|  \Delta_Y  \pi^* \FF_N \right| \leq C.$$

Also we can assume that $\pi^* \left( |S_D|^{2N} _{h_D} \right) \omega_Y \leq f^* \chi$ for sufficiently large $N$. After repeating the above estimates, we have for any $k \geq 0$, there exists sufficiently large $N$ and $C_{k, N}$ such that

$$ \left| \left( \Delta_Y \right)^{ k } \left( \pi^* \FF_N \right) \right| \leq C_{k, N}.$$

By standard elliptic estimates, $\pi^* \FF_N$ is uniformly bounded in $C^k$ if we choose $N$ sufficiently large, therefore $\pi^* F$ can have at worst pole singularities.

\qed
\end{proof}

Now let us recall some facts on the Weil-Petersson metric on the
moduli space $\mathcal{M}$ of polarized Calabi-Yau manifolds of
dimension $n-\kappa$ . Let $\mathcal{X}\rightarrow \mathcal{M}$ be
a universal family of Calabi-Yau manifolds.
Let $(U; t_1, ..., t_{\kappa})$ be a local holomorphic coordinate chart
of $\mathcal{M}$, where $\kappa=\dim {\mathcal M}$. Then each
$\frac{\partial}{\partial t_i}$ corresponds to an element
$\iota(\frac{\partial}{\partial t_i})\in H^1(\mathcal{X}_t,
T_{\mathcal{X}_t})$ through the Kodaira-Spencer map $\iota$. The
Weil-Petersson metric is defined by the $L^2$-inner product of
harmonic forms representing classes in $H^1(\mathcal{X}_t,
T_{\mathcal{X}_t})$. In the case of Calabi-Yau manifolds, as first shown in
\cite{Ti4}, it can be
expressed as follows. Let $\Psi$ be a nonzero holomorphic
$(n-\kappa,0)$-form on the fibre $\mathcal{X}_t$ and
$\Psi\lrcorner\iota(\frac{\partial}{\partial t_i})$ be the
contraction of $\Psi$ and $\frac{\partial}{\partial t_i}$. Then
the Weil-Petersson metric is given by
\begin{equation}
\left( \frac{\partial}{\partial t_i}, \frac{\partial}{\partial
\bar{t_j}} \right)_{\omega_{WP}}=\frac{\int_{\mathcal{X}_t}\Psi\lrcorner\iota(\frac{\partial}{\partial
t_i})\wedge\overline{\Psi\lrcorner\iota(\frac{\partial}{\partial
t_i})}}{\int_{\mathcal{X}_t}\Psi\wedge\overline{\Psi}}.
\end{equation}
One can also represent $\omega_{WP}$ as the curvature form of the
first Hodge bundle $f_*
\Omega^{n-\kappa}_{\mathcal{X}/\mathcal{M}}$ (cf. \cite{Ti4}). Let $\Psi$ be a
nonzero local holomorphic section of $f_*
\Omega^{n-\kappa}_{\mathcal{X}/\mathcal{M}}$ and one can define
the hermitian metric $h_{WP}$ on $f_*
\Omega^{n-\kappa}_{\mathcal{X}/\mathcal{M}}$ by
\begin{equation}
|\Psi_t|^2_{h_{WP}}=\int_{\mathcal{X}_t}\Psi_t\wedge\overline{\Psi_t}.
\end{equation}
Then the Weil-Petersson metric is given by
\begin{equation}
\omega_{WP}=\ric(h_{WP}).
\end{equation}

\noindent The Weil-Petersson metric can also be considered a canonical hermitian metric on the
dualizing sheaf
$f_*(\Omega^{n-\kappa}_{X/X_{can}})=(f_{*1}\mathcal{O}_X)^{\vee}$
over $X_{can}^{\circ}$.

 Let $X$ be an $n$-dimensional algebraic manifold.
 Suppose its canonical line bundle $K_X$ is semi-positive and
 $0<\kappa=\kod(X)<n$. Let $X_{can}$ be the canonical model of $X$. We define a canonical hermitian
 metric $h_{can}$ on $f_*(\Omega^{n-\kappa}_{X/X_{can}}) $ in the way that
for any smooth $(n-\kappa,
 0)$-form $\eta$ on a nonsingular fibre $X_y$,
\begin{equation}
|\eta|^2_{h_{WP}}=\frac{\eta\wedge\bar{\eta}\wedge\chi^{\kappa}}
{ \Theta \wedge\chi^{\kappa}}
=\frac{\int_{X_y}\eta\wedge\bar{\eta}}
{\int_{X_y} \Theta }.
\end{equation}

\begin{definition} \label{hkemetric} Let $X$ be an $n$-dimensional  algebraic manifold with semi-ample canonical line bundle $K_X$. Suppose $0 < \kod(X) < n$ and so $f: X \rightarrow X_{can}$ is a holomorphic fibration of Calabi-Yau manifolds. A closed positive $(1,1)$-current $\omega$ on $X_{can}$ is called a canonical metric if it satisfies the following.

\begin{enumerate}

\item $f^* \omega \in -2 \pi c_1 (X) $.

\item $\omega$ is smooth outside a subvariety of  $ X_{can} $ and $\left( f^*\omega \right)^{\kappa} \wedge \Theta$ is continuous on $X$.

\item $\ric(\omega) = -\ddbar \log \omega^{\kappa}$ is well-defined on $X$ as a current and  on $\Xcr $

\begin{equation}
\ric(\omega) = -\omega + \omega_{WP}.
\end{equation}

\end{enumerate}

\end{definition}

\begin{definition} \label{HKEvol}

Suppose $\omega_{can} $ is a canonical metric on $X_{can}$. We define the canonical volume form $\Omega_{can}$ on $X$ to be

\begin{equation}
\Omega_{KE} = \left(f^* \omega_{KE} \right)^{\kappa} \wedge \Theta.
\end{equation}

\end{definition}


\subsection{Existence and uniqueness}


The main goal of this section is to prove the existence and uniqueness for canonical metrics on the canonical models.

\begin{theorem} \label{HKEthm} Let $X$ be an $n$-dimensional  algebraic manifold with semi-ample canonical line bundle $K_X$. Suppose $0 < \kappa = \kod(X) < n$.  There exists a unique  canonical metric on $X_{can}$.

\end{theorem}

We will need the following theorem of solving singular Monge-Amp\`ere equation to prove Theorem \ref{HKEthm}.

\begin{theorem}\label{MAthm} There exits a
unique solution $\varphi \in  PSH (\chi)\cap
C^{0}(X_{can})\cap C^{\infty} (\Xcr ) $ to the
following Monge-Amp\`ere equation on $X_{can}$
\begin{equation}\label{MAeqn}
(\chi+\ddbar \varphi )^{ \kappa } =F e^{ \varphi } { \chi^{\kappa} }.
\end{equation}

\end{theorem}

\bigskip
\noindent{\bf Proof of Theorem \ref{HKEthm} } ~~ We will prove Theorem \ref{HKEthm} by assuming Theorem \ref{MAthm}.

\noindent Let $\varphi$ be the solution in Theorem \ref{HKEthm} and $\omega = \chi + \ddbar \varphi$.

\begin{enumerate}

\item

$$f^*\omega = f^*\chi + \ddbar f^* \varphi \in - c_1(X).$$
and it proves $1.$ in Definition \ref{hkemetric}.

\item

By Theorem \ref{MAthm}, $\omega$ is smooth on $ \Xcr $ and
$$\left( f^* \omega \right) ^{\kappa} \wedge \Theta = \Omega e^{f^*\varphi}$$
is continuous since $f^*\varphi$ is continuous on $X$ and $\Omega$ is a smooth volume form. This proves $2.$ in Definition \ref{hkemetric}

\item

Then
$$\ric( \omega) = -\ddbar \log \omega^{\kappa} =- \ddbar \log \chi^{\kappa} -\ddbar \log F - \ddbar \varphi$$
is well-defined as a current on $X_{can}$.

Calculate on $\Xcr$

\begin{eqnarray*}
&& \ddbar \log \chi^{\kappa} + \ddbar \log F + \ddbar \varphi \\
&=& \ddbar \log \chi^{\kappa} + \ddbar \log \left( \frac{ \Omega } {\Theta \wedge \chi^{\kappa}} \right) + \omega -\chi\\
&=& \omega + \left( - \ddbar \log \left( \Theta \wedge \chi^{\kappa} \right) + \ddbar \log \chi^{ \kappa }  \right) \\
&=& \omega - \omega_{WP}.
\end{eqnarray*}

Therefore

$$ \ric(\omega) = -\omega + \omega_{WP}.$$

So we have proved $3.$ in Definition \ref{hkemetric}.

\end{enumerate}

Now we will prove the uniqueness of the  canonical  metric.

Let $\omega = \chi + \ddbar \varphi $ be a  canonical  metric on $X_{can}$. Then by the equation for the canonical metric, we have on $\Xcr$
$$ \ddbar \log \left( \frac{ \omega^{\kappa} }{\chi^{\kappa} } \right) = \ddbar \log \left( \frac{\Omega} { \Theta \wedge \chi ^{\kappa} } e^{ \varphi } \right) .$$
Let  $\xi = \left( \frac{ \omega^{\kappa} }{\chi^{\kappa} } \right) \left( \frac{\Omega} { \Theta \wedge \chi ^{\kappa} } e^{ \varphi } \right)^{-1}$. Then on $\Xcr$ we have
$$\ddbar \log  \xi = 0. $$
On the other hand,
$$f^* \xi = \frac{ \omega ^{\kappa} \wedge \Theta } {\Omega} e^{-\varphi}$$
extends to a strictly positive continuous function on $X$. Since $f$ is a holomorphic map, $\xi$  extends to a continuous function on $X_{can}$. Let $\pi: X_{can}' \rightarrow X_{can} $ be a resolution of $X_{can}$. Then $\pi^* \xi$ is continuous on $X_{can}'$ and $\ddbar \log \pi^* \xi = 0$ so that  $\pi^* \xi= constant$ on $X_{can}'$. Therefore $\xi= constant>0 $ on $X_{can}$ and   $\varphi' = \varphi+ \log \xi \in PSH(\chi) \cap C^0(X_{can}) $ solves the Monge-Amp\`ere equation (\ref{MAeqn}). The uniqueness of the solution for  (\ref{MAeqn})implies the uniqueness of the  canonical  metric.

\qed


\begin{corollary}

 Let $\omega_{can}$ be a  canonical  metric on $X_{can}$ and $\Omega_{can}$ the  canonical  volume form on $X$. Then

\begin{equation}
f^* \omega_{can} = \ddbar \log \Omega_{can}.
\end{equation}

\end{corollary}

\begin{proof}

Let $\omega_{can} = \chi + \ddbar \varphi $, where $\varphi$ be the solution of the Monge-Amp\`ere equation (\ref{MAeqn}) in Theorem \ref{MAthm}. Then
$$\Omega_{can} = \left( f^*\omega_{can} \right)^{\kappa} \wedge \Theta = \Omega e^{ f^* \varphi}$$
and so
$$ \ddbar \log \Omega_{can} = \ddbar \log \Omega + \ddbar f^*\varphi = f^* \omega_{can}.$$

\qed
\end{proof}

\bigskip


\noindent{\bf Proof of Theorem \ref{MAthm}}

\medskip

\noindent{\it Step 1.} {\bf Approximation}

\medskip

\noindent Let $\pi: Y \rightarrow X_{can}$ be a resolution of
singularities such that $E=\pi^*(X_{can}\backslash
X_{can}^{\circ})$ is a divisor with simple normal crossings.  Let
$\hchi=\pi^*\chi$ and $\hF=\pi^* F$. Then $\hchi$ is a closed
semi-positive $(1,1)$-form on $Y$ and
$\int_{ Y} \hchi^{\kappa}=\int_{X_{can}} \chi^{\kappa}>0$,
i.e., $\hchi$ is big.  We will consider the following
Monge-Amp\`ere equation on $Y$

\begin{equation}\label{maequation}
\frac{(\hchi+\ddbar \hvarphi )^{\kappa}}
{\hchi^{\kappa}}=\hF e^{\hvarphi}.
\end{equation}

Since $\hchi$ is a big semi-positive closed $(1,1)$-form and so $\hat{L} =  \frac{1}{m} \pi^* \mathcal{O} (1)$ is a semi-positive big line bundle on $Y$. By Kodaira's lemma, there exists a divisor $D$ such that for any $\epsilon>0$,  $[L]- \epsilon [D]$ is an ample divisor on $Y$. Let $S_D$ be the defining section of $D$ and choose a smooth hermitian metric $h_D$ on the line bundle associated to $[D]$ such that

$$ \hchi + \epsilon \ddbar \log h_D >0 . $$

Fix $\epsilon_0 >0 $ and  define a K\"ahler form $\omega_0 = \hchi + \epsilon_0 \ddbar \log h_D >0$.

$\hF \in L^{1+\epsilon}(Y, \hchi^{\kappa} )$ for some $\epsilon>0$ since  $F \in L^{1+\epsilon}( X_{can} , \chi^{\kappa} )$ for some $\epsilon>0$. Then for each $k>0$ there exists a family of  functions $ \{ F_j \}_{j=1}^{\infty}$ on $Y$  satisfying the following.

\begin{enumerate}

\item $ F_j \in C^{4} (Y)$ for all $j$ and  $ F_j \rightarrow \hF$ in $L^{1+\epsilon}(Y, \hchi^{\kappa} )$ as $ j \rightarrow \infty $.

\item There exists $C>0$ such that $ \log F_j \geq - C$ for all $j$.


\item There exist $\lambda $,  $C > 0 $ such that for all $j$

$$ \left| \left|  |S_E|_{h_E}^{2\lambda }  F_j \right| \right|_{ C^{2}( Y )} \leq C,$$
where $S_E$ is a defining section of $E$ and $h_E$ is a fixed smooth hermitian metric on the line bundled associated to $[E]$.
\end{enumerate}

For example, we can choose $\hF_j$ to be defined by

$$  F_j = \exp \left(  \left( \frac{ |S_E|_{ h_E}^{2 \alpha} } { j^{-1} + |S_E|_{ h_E}^{2\alpha} } \right) \log \hF \right)$$
for sufficiently large $\alpha >0 $.

We also choose a K\"ahler form $\omega_0$ and let $\displaystyle \chi_j = \hchi +\frac{1}{j}~ \omega_0$.
We consider the following Monge-Amp\`ere equation

\begin{equation}\label{MAfamily}
\frac{( \chi_j + \ddbar \varphi_j )^{\kappa } }{ \left( \chi_j \right) ^{ \kappa } }= F_j
e^{\varphi_j }
\end{equation}
for sufficiently large $\alpha$.

By Yau's solution to the Calabi conjecture, for each $ j $, there exists a unique solution $\varphi_j \in C^{3}(Y) \cap C^{\infty} ( Y \setminus
E )$ solving (\ref{MAfamily}).  We will derive uniform estimates for $\varphi_j$.

\bigskip

\noindent{\it Step 2.} {\bf Zeroth order estimates}

\begin{proposition} There exists $C>0$ such that for all $ j$,

\begin{equation}
\left| \left| \varphi_j \right| \right| _{ L^{\infty} (Y ) } \leq C.
\end{equation}

Furthermore, there exists a unique solution  $\varphi_{\infty} \in PSH (Y, \hchi) \cap C^0 (Y)$ solving the Monge-Amp\`ere equation (\ref{maequation}) such that
\begin{equation}
\varphi_{j} \rightarrow \varphi_{\infty}
\end{equation}
in $L^1(Y, \omega_0^{\kappa} )$ as $j\rightarrow \infty$.

\end{proposition}

\begin{proof}   By Yau's theorem, for each $j$, there exists a smooth solution $\varphi_j$ solving the Monge-Amp\`ere equation (\ref{maequation}).  We first derive a uniform upper bound for $\varphi_j$. Suppose that $\varphi_j$ achieves its maximum at $y_0$ on $Y$. Then applying the maximum principle, we have

$$\varphi_j \leq \varphi_j(y_0) \leq \frac{1}{ F_j (y_0)}  \leq \sup_Y  \left( \frac{1}{ F_j} \right) \leq C.$$

Let $ \Omega_0 $ be a smooth nowhere vanishing volume form on $Y$. We have to verify that $  F_j  \left( \frac{ \hchi_j^{\kappa} }{\hOmega} \right)  $ is uniformly bounded in $L^{1+\epsilon }(Y, \hOmega) $ for some $\delta>0$ for all $j$.

\begin{eqnarray*}
\left| \left| e^{\varphi_j} F_j  \left(\frac{ \chi_j^{\kappa} } {\Omega_0}  \right) \right| \right|^{1+ \epsilon }_{L^{1+ \epsilon }(Y, \Omega_0)}
=  \int_Y e^{(1+ \epsilon ) \varphi_j} F_j^{1+ \epsilon } \left( \frac{ \chi_j^{\kappa} }{\Omega_0} \right) ^{ \epsilon }  \chi_j^{\kappa} \leq  C \int_Y F_j^{1+\epsilon }   \chi_j^{\kappa} = C \left| \left| F_j  \right| \right|^{1+ \epsilon }_{L^{1+ \epsilon }(Y, \Omega_0)} .
\end{eqnarray*}
Therefore $ \left| \left| \varphi_j \right| \right|_{ L^{\infty} (Y )} $ is uniformly bounded by Theorem \ref{zhang}, since $ \left| \left| F_j  \right| \right|^{1+ \epsilon }_{L^{1+ \epsilon }(Y, \Omega_0)} $ is uniformly bounded.

Also Theorem \ref{zhang} gives a unique solution $\varphi_{\infty} \in PSH( Y, \hchi )\cap C^0 (Y)$ solving the Monge-Amp\`ere equation \ref{maequation}. By the uniqueness of such $\varphi_{\infty}$, we have $\varphi_j \rightarrow \varphi_{\infty}$ in $L^1$ by standard potential theory.


%

%
%
%

\qed
\end{proof}

\bigskip
\bigskip

\noindent{\it Step 3.} {\bf Second order estimates}

\bigskip
\noindent We now apply the maximum principle and prove a second
order estimate for $\varphi_j$ by using a modified argument in
\cite{Ya2}. Note that Tsuji used a similar trick in \cite{Ts1} for the second-order estimate.

Let $\omega_j = \chi_j + \ddbar \varphi_j$, $\Delta_0$ and $\Delta_j$ be the Laplace operator associated to $\omega_0$ and $\omega_j$.
The following lemma is
proved by standard calculation.

\begin{lemma} There exists a uniform constant $C>0 $ only depending on $\omega_0$ such that on $Y \setminus \left(  E \cup D \right) $
$$\Delta_j \log \tr_{\omega_0}(\omega_j)\geq
 -C \left( 1 + \tr_{\omega_j} (\omega_0) +  \frac{1 + \left| \Delta_0 \log F_j  \right|}
{\tr_{ \omega_0} ( \omega_j )} \right) . $$
\end{lemma}

\begin{theorem}  There exist  $\alpha, ~ \beta,~
C>0 $ such that for all $j$ and $z\in Y \setminus \left(  E\cup D \right) $
\begin{equation}\label{2nd}
\tr_{ \omega_0 } ( \omega_j ) ( z ) \leq \frac{C}{ |S_E|_{h_E}^{2\alpha} | S_D |_{h_D}^{2\beta} }
\end{equation}

\end{theorem}

\begin{proof}

Define
$$
\Phi_j = \varphi_j - \epsilon_0 \log|S_D|^2_{h_D}$$
and
$$ H_{j}  =  \log
\left(  |S_E |^{2 \alpha }_{h_E} \tr_{\omega_0 } ( \omega_j ) \right)
-A \Phi_j
$$
for some constants
$\alpha, ~A>0$ to be determined later.


First calculate on $ Y \setminus \left( E\cup D \right) $
\begin{eqnarray*}
&&\Delta_j  H_{ j}\\
&&\\
&=&
\Delta_j \log \tr_{ \omega_0 } ( \omega_j )  - A \Delta_j \Phi_j +
 \alpha \Delta_j \log |S_E| _ { h_E } ^2\\
&&\\
&=&
\Delta_j \log \tr_{\omega_0 } ( \omega_j ) - A
\tr_ { \omega_j } \left( \omega_j - \omega_0 \right)-
\alpha  \tr_{ \omega_j } ( \ric( h_E)) \\
&&\\
&\geq&  (A - C_1) \tr_{\omega_j }( \omega_0 )-\alpha \tr_{\omega_j }
( \ric ( h_E ) )-  \frac{ C_1\left(  1+ \left| \Delta_0  \log F_j \right| \right)  }
{ \tr_{\omega_0 }(\omega_j) } - \frac{C_1} {\tr_{\omega_0 }( \omega_j)} -  \kappa A - C_1
\end{eqnarray*}
Choose a sufficiently large $\alpha>0$, such that  there exists a
constant $C_2 >0$ with the following:
\begin{eqnarray*}
\frac{  \left| S_E \right|^{2 \alpha}_{h_E} \left( 1 + \left| \Delta_0 \log F_j \right| \right)}  { \tr_{\omega_0} (\omega_j)} 
%
%
%
%
&\leq & \frac { C _2 } { \tr_{\omega_0 } (\omega_j )  }
\end{eqnarray*}
and

$$ |S_E|_{h_E}^{2\alpha} F_j \leq C_2.$$
Applying  the elementary inequality
\begin{eqnarray*}
\tr_{ \omega_j } ( \omega_0 ) &\geq&  C_3
\left (\tr_{ \omega_0 } ( \omega_j ) \right)^{ \frac{1} { \kappa - 1 } }
\left( \frac{ \omega_0^{ \kappa } } { \omega_j ^{ \kappa } } \right)
^{\frac{1}{\kappa-1}}  \\
&&\\
&=& C_3  \left(  \frac{ e^{ - \varphi_j  } } { F_j }  \tr_{ \omega_0 } ( \omega_j )   \left( \frac{ \omega_0^{ \kappa } } { \chi_j ^{ \kappa } } \right)  \right)^{ \frac{1} { \kappa - 1 } }  \\
&&\\
&\geq& C_4 \left(  \left| S_E \right|_{h_E}^{2\alpha} \tr_{\omega_0} ( \omega_j ) \right)^{ \frac{1} { \kappa - 1 } }
\end{eqnarray*}

We can always choose $A$ sufficiently large such that
$$
\Delta_j H_j   \geq   C_5 A  \left(  \left| S_E \right|_{h_E}^{2\alpha} \tr_{\omega_0} ( \omega_j ) \right)^{ \frac{1} { \kappa - 1 } } - \frac { C _4 |S_E|^{4\alpha}_{h_E} } {    \left| S_E \right|_{h_E}^{2\alpha} \tr_{\omega_0} ( \omega_j )  } - C_6 A.
$$

For any $j$,  suppose $$\sup_{z \in Y }
H_j (z)=H_j (z_0) $$ for some $z_0\in Y \setminus \left( E \cup D \right) $  since
$H_j = - \infty $ along $ E \cup D $. By the maximum principle, $\Delta_j
H_j (z_0)\leq 0$. By straightforward calculation, there exists $C_7>0$ such that  %
$$ \left. \left| S_E \right| _{h_E}^{2\alpha}  \tr_{\omega_0 } (\omega_j) \right|_{z = z_0}  \leq C_7.$$
Hence there exists a uniform constant $C_8>0$ such that

$$H_j \leq  H_j(z_0) = \left. \log \left( |S_E|^{2\alpha} \tr_{\omega_0} (\omega_j) \right) \right|_{z=z_0}- \varphi_j (z_0) + \epsilon_0 \log |S_D|^{2}_{h_D} (z = z_0) \leq C_8.$$
The theorem is proved since

$$ \tr_{\omega_0 } (\omega_j ) \leq |S_E|^{-2\alpha}_{h_E} \exp \left( H_j + A \Phi_j \right) \leq \frac{ C_9 } {  |S_E|^{2\alpha}_{h_E}  |S_D|^{2\beta}_{h_D}  } . $$

\qed
\end{proof}

\bigskip
\bigskip


\noindent {\it Step 4. } {\bf $C^k$  estimates}

\bigskip

\noindent From the second order estimates, the Monge-Amper\`e
equation (\ref{MAfamily}) is uniformly elliptic on any compact set of
$Y \setminus (E\cup D) $. The $C^k$-estimates are local estimates and can be
derived by standard Schauder estimates and elliptic estimates. Therefore for any compact subset $K$ of  $Y \setminus (E\cup D) $, there exist constants $C_{K, R}$ such that

$$ \left| \left| \varphi_j \right| \right|_{C^R( K)} \leq C_{R, K}.$$

\bigskip
\bigskip

\noindent {\it Step 5. } {\bf Proof of Theorem \ref{MAthm}}

\bigskip

By taking a sequence, we can assume $\varphi_j \rightarrow
\varphi_{\infty} \in L^{1 } (Y, \omega_0^{\kappa}  )$. Therefore $$\varphi_{\infty} \in C^0(Y) \cap C^{\infty} ( Y\setminus (E\cup D)).$$ On the other hand, one can choose different divisors $D$ such that $[\hat{\chi}]-\epsilon [D]>0$ for all sufficiently small $\epsilon>0$ and the intersection of such divisors is contained in $E$. Therefore $\varphi_{\infty} \in C^{\infty} (Y \backslash E)$.
Each fibre of the resolution $\pi$ is connected and so   $\varphi_{\infty}$ is constant along the fibre since $\varphi_{\infty} \in PSH(Y, \hchi) \cap C^0( Y) $ and $\hchi\geq 0$.
Therefore $\varphi_{\infty}$ descends to a solution $\varphi \in PSH(X_{can}, \chi) \cap C^0( X_{can}) $ solving equation (\ref{MAeqn}) as in Theorem \ref{MAthm}. Furthermore, $\varphi$ is smooth on $X_{can}^{\circ}$. This completes the proof of Theorem \ref{MAthm}.

 \qed

\bigskip


\section{Canonical measures on algebraic manifolds of non-negative Kodaira dimension}

\subsection{Canonical measures on surfaces of non-negative Kodaira dimension }

Let $X$ be a K\"ahler surface of positive Kodaira dimension and
$X_{min}$ be its minimal model derived by $\pi: X \rightarrow
X_{min}$ contracting all the $(-1)$-curves $E=\cup_i E_i$.

If $\kod(X)=2$, $X$ is a surface of general type. Let $\Phi:
X_{min}\rightarrow X_{can}$ be holomorphic canonical map from
$X_{min}$ to its canonical model $X_{can}$ with possible orbifold
singularities. There exists a unique smooth orbifold
K\"ahler-Einstein metric $\omega^{\dagger}$ on $X_{can}$. We define

$$\omega_{KE}=(\pi\circ\Phi)^*\omega^{\dagger}+\ddbar\log |E|^2.$$

Choose a smooth positive $(1,1)$-form $\omega\in -c_1(X_{min})$
and a smooth volume form $\Omega$ with $\ddbar\log
\Omega=\omega$. Then $\Phi^*\omega^{\dagger}=\omega+\ddbar
\varphi$ for some $\varphi\in C^0(X_{min})$ satisfying the
following Monge-Amp\`ere equation

$$\frac{(\omega+\ddbar \varphi)^2}{\Omega}=e^{\varphi}.$$
Furthermore, $\varphi$ is smooth outside the exceptional locus of
the pluricanonical system. Let $\Omega_{KE}=\pi^*(e^{\varphi} \Omega )$ the pullback of the unique holomorphic  K\"ahler-Einstein volume from its canonical model. It is a continuous measure on $X$ vanishing exactly on $E$ of order one. Then  $ \omega_{KE} = Ric(\Omega_{KE})$ on $X$ and $\omega_{KE}^2= \Omega_{KE}$ on $X\setminus E$.

\begin{proposition} Let $X$ be a K\"ahler surface of general type. Then

\begin{enumerate}

\item $\omega_{KE}\in -c_1(X)$,

\item $\ric(\omega_{KE})=-\omega_{KE}$ on $X \backslash E$,

\item $h_{KE}= \Omega_{KE}^{-1}$ is an analytic Zariski decomposition for $K_X$.

\end{enumerate}

\end{proposition}

If $\kod(X)=1$, $X$ is an elliptic surface. Let $\Phi:
X_{min}\rightarrow X_{can}$ be the holomorphic pluricanonical fibration from
$X_{min}$ to its canonical model $X_{can}$. By Theorem \ref{HKEthm} in Section 3.2, there exist a canonical metric $\omega^{\dagger}$ on $X_{can}$   and a canonical measure $\Omega^{\dagger} $ on $X_{min}$ such that $\Phi^* ( \omega^{\dagger} ) \in - c_1 (X_{min})$ and
$\ric(\Omega^{\dagger}) = \Phi^*( \omega^{\dagger} )$. Let $\Omega_{can} = \pi^* ( \Omega^{\dagger})$ and $\omega_{can} = \ric(\Omega_{can}$.

\begin{proposition} Let $X$ be a K\"ahler surface of Kodaira dimension $1$. Then

\begin{enumerate}

\item $\omega_{can}\in -c_1(X)$,

\item $\omega_{can} = (\Phi\circ\pi)^* (\omega^{\dagger})$ on $ X\backslash E$.

\item $h_{can} = \Omega_{can}^{-1}$ is an analytic Zariski decomposition for $K_X$.

\end{enumerate}

\end{proposition}


\subsection{Ricci-flat metrics on K\"ahler manifolds of zero Kodaira dimension}

There have been many interesting results on singular Ricci-flat metrics. In \cite{EyGuZe1}, singular Ricci- flat metrics are studied on normal K\"ahler spaces. In \cite{To}, singular Ricci-flat metrics are derived as the limit of smooth Ricci-flat K\"ahler metrics along certain degeneration of K\"ahler classes. In this section, we construct singular Ricci-flat metrics on algebraic manifolds of Kodaira dimension $0$ as an immediate application of Theorem \ref{zhang}.

Let $X$ be an $n$-dimensional algebraic manifold of Kodaira dimension $0$. Suppose $ L \rightarrow X$ is a holomorphic line bundle such that $L$ is big and semi-ample. There exists a big smooth semi-positive $(1,1)$-form $\omega \in c_1 (L)$. Let $\eta \in H^0(X, K_X^m)$ be the holomorphic $m$-tuple $n$-form for some $m\in N(K_X)$. Let $$\Omega= \frac{ [\omega]^n} {\int_X \left( \eta\otimes \overline{\eta} \right)^{\frac{1}{m}} } \left( \eta\otimes \overline{\eta} \right)^{\frac{1}{m}}$$ be a smooth $(n,n)$-form on $X$ and so $\int_X \Omega = [\omega]^n$. $\Omega$ is independent of the choice of $m\in N(K_X)$ and $\eta$ because the Kodaira dimension of $X$ is $0$. $\Omega$ is unique up to a scalar multiplication and it can be degenerate because $K_X$ is not necessarily nef.

Consider the following degenerate Monge-Amp\`ere equation
\begin{equation}\label{cy}
(\omega + \ddbar \varphi)^n = \Omega.
\end{equation}
By Theorem \ref{zhang}, there exists a continuous solution $\varphi$ t to equation (\ref{cy}) unique up to a constant. Let $\omega_{CY} = \omega + \ddbar \varphi$.

\begin{proposition} $\omega_{CY}$ is the unique closed semi-positive $(1,1)$-current in $c_1(L)$ with continuous local potential such that $\omega_{CY}^n = \Omega$ and therefore outside the base locus of the pluricanonical system
$$\ric(\omega_{CY}) = 0.$$
Furthermore, $\omega_{CY}$ is smooth on a Zariski open set of $X$.

\end{proposition}


\subsection{K\"ahler-Einstein metrics on algebraic manifolds of general type}

In this section, we will prove Theorem B.1. Let $X$ be an $n$-dimensional nonsingular algebraic variety of general type.

We choose a sequence of resolution of indeterminacies of the pluricanonical systems $\Phi_{m!}$

\begin{equation}
\begin{diagram}
\node{X}      \node{X_{m_0}}  \arrow{w,t}{\pi_{m_0}}  \node{X_{m_0+1}}  \arrow{w,t}{\pi_{m_0+1}} \node{ \cdot\cdot\cdot} \arrow{w,t}{\pi_{m_0+2}} \node{X_m} \arrow{w,t}{\pi_m}  \node{X_{m+1}} \arrow{w,t}{\pi_{m+1}} \node{\cdot\cdot\cdot} \arrow{w,t}{\pi_{m+2}}
\end{diagram}
\end{equation}

for $m_0$ sufficiently large, such that

\begin{enumerate}

\item

$$(\bar{\pi}_m)^* (m! K_X) = L_m + E_m,$$
where $\bar{\pi}_m= \pi_m\circ\pi_{m-1} \circ...\circ\pi_{m_0}$.

\item

$$E_m = \sum_{j} c_{m, j} E_{m, j} $$ is the fixed part of $\left| \bar{\pi}_m^*  ( m!K_X) \right|$ with each $E_{m, j}$ being a divisor with simple normal crossings.

\item $L_m$ is a globally generated line bundle on $X_m$.

\end{enumerate}

Let $ \{ \sigma_{m, j} \} _{ j = 0} ^{d_m}  $ be a basis of $ H^0 (X , m! K_X )$ and
 $ \{ \zeta_{m, j} \} _{ j = 0} ^{d_m}  $ be a basis of $ H^0 (X_m , L_m )$ such that $$  \pi_m ^* \sigma_{ m, j } = \zeta_{m, j}  E_m  .$$

We can consider $ | \sigma_{m, j} |^{\frac{2}{m!}} $ as an smooth $(n, n)$-form on $X$ as $ \sigma_{m, j} \in m! K_X $.   Let $\Omega_m = \left( \sum_{j=0}^{d_m} \left| \sigma_{m, j} \right|^2 \right)^{\frac{1}{m!} }$ and then $\pi_k^* \Omega_m $ is a smooth and possibly degenerate volume form on $X_k$.    For simplicity we also use $\Omega_m$ for $(\pi_{k})^* (\Omega_m) $ for all $k\geq m_0$.

The following lemma is obviously by the construction of $\Omega_m$.

\begin{lemma} On $X$ and so on $X_k$ for $k\geq m_0$,

\begin{equation}
\frac{\Omega_m}{\Omega_{m+1}} < \infty.
\end{equation}

\end{lemma}

Define a smooth closed semi-positive $(1,1)$-form $\omega_{m}$ on $X_m$ by
$$\omega_{m}= \frac{1}{m!} \ddbar \log \left( \sum_{j= 0 } ^{ d_m} | \zeta_j |^2 \right).$$
Obviously $\ddbar \log \left( \sum_{j= 0 } ^{ d_m} | \zeta_j |^2 \right)$ is the pullback of the Fubini-Study metric on $\mathbf{CP}^{d_m}$ from the linear system $| L_m |$.

\begin{theorem}\label{ke_m} Let $R_m$ be the exceptional locus of the linear systems associated to $L_m$.
There exists a unique solution $\varphi_m\in C^0(X_m) \cap C^{
\infty }( X_m \setminus \left(\cup_{j} E_{m, j} \cup R_m
\right))$ to the following Monge-Amp\`ere equation

\begin{equation}\label{gt1}
\left(\omega_{m}+\ddbar \varphi \right)^n =
e^{\varphi}  \Omega_m.
\end{equation}
Therefore
$\omega_{KE, m}=\omega_{m}+\ddbar \varphi_m$ is a
K\"ahler-Einstein current on $X_m$ satisfying

\begin{enumerate}

\item $\omega_{KE, m}$ is a positive current on $X_m$ and strictly
positive on $X_m \setminus \left(\cup_{j} E_{m, j} \cup R_m
\right) $,

\item $\ric(\omega_{KE, m})= -\omega_{KE, m}$ on $X_m \setminus
\left(\cup_{j} E_{m, j} \right) $.

\end{enumerate}

\end{theorem}

\begin{proof} Let $F_m = \frac{ \Omega_m } {\left( \omega_{m} \right)^n  }$. $F_m$ has at worst pole singularities on $X_m$ and
$$ \int_{X_m} \left( F_m \right) ^{ 1 + \epsilon } ( \omega_{m} )^n = \int_{X_m} (F_m)^{\epsilon} \Omega_m < \infty .$$ By Theorem \ref{zhang}, there exists a unique $ \varphi_m \in PSH( X_m, \omega_{m} ) \cap C^0(X_m)$ solving the equation (\ref{gt1}).

Also $L_m$ is an semi-ample line bundle and furthermore it is big. Then by Kodaira's Lemma, there exists a divisor $[F_m]$ such that
$$ [ L_m ] - \epsilon
[F_m] = \pi_m^* [K_X] - \sum_{j} \frac{c_{m, j}}{m} [E_{m, j}]- \epsilon
[F_m]    $$ is ample for sufficiently small $\epsilon>0$. By a similar argument in the proof of Theorem \ref{MAthm}, we can show that $\varphi_m \in C^{\infty} \left( X_m \setminus \left(\cup_{j} E_{m, j} \cup R_m
\right) \right)$.

\qed
\end{proof}

\begin{corollary} $ e^{\varphi_m} \Omega_m$ descends from $X_m$ to a continuous measure on $X$ and $h_m = e^{-\varphi_m} \Omega_m^{-1}$ is a singular hermitian metric on $X$ with its curvature $\Theta_{h_m} \geq 0$. Furthermore, on $X\backslash Bs(| m_0 K_X |)$,

\begin{equation}\label{ke_mx}
(\ddbar\log \Omega_m + \ddbar \varphi_m )^n = e^{\varphi_m}  \Omega_m.
\end{equation}

\end{corollary}

\begin{proof} $e^{\varphi_m} \Omega_m$ can be pulled back as a continuous volume form on $X \backslash Bs( m! K_X)$. On the other hand, $\varphi_{m}$ is uniformly bounded in $L^{\infty}(X)$, also $\Omega_m$ is smooth on $X$ and  vanishes exactly on  $Bs( |m! K_X| )$. Therefore $ e^{\varphi_m} \Omega_m$ is continuous on $X$.

$Bs( | m! K_X| )$ is a complete closed pluripolar set on $X$ using $H^0(X, m!K_m)$. Since $\log \Omega_m +  \varphi_m$  is uniformly bounded above and $\ddbar \log \Omega_m + \ddbar \varphi_m$ is a positive closed $(1,1)$ current on $X\backslash Bs( | m_0 K_X | )$, $\ddbar \log \Omega_m + \ddbar \varphi_m$ extends to a positive closed current on $X$ using local argument.

Equation (\ref{ke_mx}) is then derived directly from Equation (\ref{gt1}).

\qed
\end{proof}

Although $\Theta_{h_m}$ is a singular K\"ahler-Einstein metric on $X$,  but without assuming finite generation of canonical rings, $(K_X, h_m)$ is not necessarily an analytic Zariski decomposition. One has to let $m$ tend to infinity in order for $e^{\varphi_m} \Omega_m$ to have the least vanishing order.

Let $D$ be an ample divisor on $X$ such that $K_X+D$ is ample.
Then there exists a hermitian metric $h_D$ on the line bundle
induced by $[D]$ such that $\omega_0-\ddbar \log h_D>0$. We can
also assume that $D$ contains the base locus of all $|m!K_X|$ for $m\geq m_0$.

\begin{lemma}\label{upbound1}  Let $\Omega_0$ be a smooth and nowhere vanishing volume form on $X$ . Then there exists a constant $C>0$ such that for each $m\geq m_0$,

$$e^{\varphi_m}\Omega_m \leq C \Omega_0.$$

\end{lemma}

\begin{proof} Let $\chi_0 = \ddbar \log \Omega_0$. Let $D_m$ be a divisor on $X$ such that on $X \backslash D_m$,  $\varphi_m$ is smooth  and $\Omega_m$ is strictly positive. Let $D$ be an ample divisor on $X$. Let $S_{D_{m}}$ be a defining function and $h_{D_m}$ be a smooth hermitian metric on the line bundle associated to $[D_{m}]$. Let $S_{D}$ be a defining function and $h_D$ be a smooth positively curved hermitian metric on the line bundle associated to $[D]$.

Let $\theta_{D_{m}} = -\ddbar \log h_{D_{m}}$ and $\theta_{D} = -\ddbar \log h_D > 0$. We also define

$$ \psi_{m, \epsilon} = \varphi_m + \log \frac{\Omega_m}{\Omega_0} + \epsilon^2 \log |S_{D_{m}}|_{h_{m}}^2 + \epsilon \log | S_D|^2_{h_D}.$$

 For simplicity, we use $\chi_0$ and $ \Omega_0$ for $(\bar{\pi}_m)^* \chi_0$ and $(\bar{\pi}_m)^* \Omega_0$. Notice that $\chi_0$ is not necessarily positive and $\Omega_0$ might vanish somewhere on $X_m$. Then outside $D_{m}$ and $D$, $\psi_{m, \epsilon}$ satisfies the following equation

 $$ \frac{ (\chi_0 + \epsilon^2 \theta_{D_{m}} + \epsilon \theta_{D} + \ddbar \log \psi_{m, \epsilon} )^n } {\Omega_0 } = |S_{D_{m}}|^{-2\epsilon^2}_{h_{D_{m}}} |S_D|^{-2\epsilon}_{h_D} e^{\psi_{m, \epsilon}}.$$

 It is easy to see that $\psi_{m, \epsilon}$ tends to $-\infty$ near $D_{m}$ and $D$, and for $\epsilon>0$ sufficiently small, $\theta_D + \epsilon \theta_{D_{m}} >0$.      By the maximum principle

$$ e^{\psi_{m, \epsilon}} \leq  \max_X \left(|S_{D_{m}}|^{2\epsilon^2}_{h_{D_{m}}} |S_D|^{2\epsilon}_{h_D}  \frac{ (\chi_0 + \epsilon^2 \theta_{D_{m}} + \epsilon \theta_{D}  )^n } {\Omega_0 } \right) $$
and
$$e^{\varphi_m}\frac{\Omega_m}{\Omega_0}  \leq |S_{D_{m}}|^{-2\epsilon^2}_{h_{D_m}} |S_D|^{-2\epsilon}_{h_D}   \max_X  \left(|S_{D_{m}}|^{2\epsilon^2}_{h_{D_{(m)}}} |S_D|^{2\epsilon}_{h_D}  \frac{ (\chi_0 + \epsilon^2 \theta_{D_{m}} + \epsilon \theta_{D}  )^n } {\Omega_0 } \right) .$$

Now we let $\epsilon$ tend to $0$, and

$$e^{\varphi_m}\frac{\Omega_m}{\Omega_0}  \leq \max_X  \left( \frac{ (\chi_0   )^n } {\Omega_0 } \right) .$$

\qed
\end{proof}

\begin{proposition} \label{mon1} $e^{\varphi_m}\Omega_m$ is increasing, that is, on $X$

\begin{equation}
e^{\varphi_m} \Omega_m \leq e^{\varphi_{m+1}} \Omega_{m+1}.
\end{equation}

\end{proposition}

\begin{proof} We shall compared $e^{\varphi_m}\Omega_m$  and $e^{\varphi_{m+1}}\Omega_{m+1}$ on $X_{m+1}$.
 Let $$U_{(m+1)}=\left\{ s \in X_{m+1}~\left|~ \frac{\Omega_{m+1}}{\Omega_m } < \infty \right. \right\}.$$
$\varphi_m$ and $\varphi_{m+1}$ are the solutions of

$$  \left( \omega_{m} + \ddbar \varphi_{m} \right) ^ n = e^{\varphi_{m+1}} \Omega_{m}.$$
and
$$  \left( \omega_{m+1} + \ddbar \varphi_{m+1} \right) ^ n = e^{\varphi_{m+1}} \Omega_{m+1}.$$
Define $$\psi = \varphi_{m+1} -\varphi_{m} + \log \frac{\Omega_{m}}{\Omega_{m+1}}$$ and $$V=\{ z\in X_{m+1} ~|~ \psi < 0 \} . $$ It is easy to see that

$$V \subset U_{m+1}$$
since both $\varphi_{m+1}$ and $\varphi_{m}$ are in $L^{\infty}(X_{m+1})$.\\

On $U_{m+1}$, $\log \left(\frac{ (E_{m+1})^{m!} }{ \left((\pi_{m+1})^*E_m\right)^{(m+1)!}}\right)$ is holomorphic and so

\begin{eqnarray*}
&&\omega_{m+1}+\ddbar\varphi_{m+1} \\
&=& \omega_{m}+ \ddbar \varphi_{m} + \ddbar \log  (\varphi_{m+1} -\varphi_{m}) + \ddbar \log \left( \frac{\left( \sum_{j=0}^{d_{m+1}} |\zeta_{m+1, j} |^2 \right)^{\frac{1}{(m+1)!}}} {\left( \sum_{j=0}^{d_m} |(\pi_{m+1})^*\zeta_{m, j} |^2 \right)^{\frac{1}{m!}}}\right)\\
&=& \omega_{m}+ \ddbar \varphi_{m} + \ddbar \psi+ \ddbar \log \left(\frac{ |E_{m+1} |^{\frac{2}{ (m+1)! }} }{ |(\pi_{m+1})^*E_m|^{\frac{2}{m!}}}\right)\\
&=& \omega_{m}+ \ddbar \varphi_{m}+ \ddbar \psi .
\end{eqnarray*}
Hence $\psi$ satisfies the following equation on $U_{m+1}$
\begin{equation}
\frac{ \left( \omega_{m}+ \ddbar \varphi_{m} + \ddbar \psi \right)^n }
 { \left( \omega_{m}+ \ddbar \varphi_{m} \right)^{n} }= e^{\psi}.
\end{equation}
By the comparison test,
\begin{eqnarray*}
&&\int_V \left( \omega_{m} + \ddbar \varphi_{m} \right)^n\\
&\leq& \int_V \left( \omega_{m} + \ddbar \varphi_{m} + \ddbar \psi \right)^n\\
&=&\int_V e^{\psi} \left( \omega_{m} + \ddbar \varphi_{m} \right)^n.
\end{eqnarray*}
Therefore $\Psi=0$ on $V$ due to the fact that it is continuous on $V$, and so on $X_{m+1}$
$$\psi\geq 0.$$
This completes the proof.

\qed
\end{proof}

\begin{proposition} There exists a measure $\Omega_{KE}$ on $X$ such that

\begin{enumerate}

\item

$\Omega_{KE} = \lim_{m\rightarrow \infty} e^{\varphi_m} \Omega_m.$

\item $(K_X, \Omega_{KE}^{-1})$ is an analytic Zariski decomposition. Furthermore, on $X$

$$\frac{\Omega_{KE}} {\Omega_0} < \infty ~~~~~{\it and }~~~~~\frac{\Psi_{X, \epsilon}}{\Omega_{KE}} < \infty.$$

\item

 $\omega_{KE} = \ddbar\log\Omega_{KE} \in - c_1(X)$ is a closed positive K\"ahler-Einstein current and on $X\backslash Bs(X, K_X)$,

$$ (\omega_{KE})^n = \Omega_{KE} ~~~~~{\it and ~ so }~~~~~ \ric( \omega_{KE} )= -\omega_{KE}.$$

\end{enumerate}

\end{proposition}

\begin{proof} By Lemma \ref{upbound1} and Proposition \ref{mon1}, we can define $\Omega_{KE}$

$$\Omega_{KE} = \lim_{m\rightarrow \infty} e^{\varphi_m} \Omega_m$$
and $\frac{\Omega_m}{\Omega_0} $ is uniformly bounded from above.   Since $\varphi_m + \log \frac{\Omega_m}{\Omega_0} \in PSH(X, \chi_0) $ and $\{\varphi_m + \log \frac{\Omega_m}{\Omega_0}\}_{m=m_0}^{\infty}$ is convergent in $L^1(X)$. Also $PSH(X, \chi_0)\cap L^1(X)$ is closed in $L^1(X)$. Therefore

$$ \lim_{m\rightarrow \infty}  \left( \varphi_m + \log \frac{\Omega_m}{\Omega_0} \right) = \log \frac {\Omega_{KE}}{\Omega_0}$$
in $PSH(X, \chi_0)\cap L^1(X)$ and $\log \frac {\Omega_{KE}}{\Omega_0} \in PSH(X, \chi_0)\cap L^1(X)$.

Let $h_{KE} = \Omega_{KE}^{-1}$ be the hermitian metric on $K_X$. By the construction of $\Omega_m$,

$$ |\sigma|^2_{h_{KE}^m} < \infty$$
for any section $\sigma\in H^0(X, mK_X)$. Therefore $\frac{\Psi_{X, \epsilon}}{\Omega_{KE}} < \infty$ and so

$$H^0(X, \OO_X(mK_X)\otimes \II(h_{KE}^m))\rightarrow H^0(X, \OO_X(mK_X))$$
is an isomorphism and $(K_X, h_{KE})$ is an analytic Zariski decomposition.

Since $\varphi_m + \log \frac{\Omega_m}{\Omega_0}$ converges uniformly on any compact set of $X\backslash Bs(X, K_X)$ to $\log \frac{\Omega_{KE}}{\Omega_0}$, we have on $X\backslash Bs(X, K_X)$,
$$(\ddbar \log \Omega_{KE} )^n = \lim_{m \rightarrow \infty}  (\ddbar \log \Omega_m + \ddbar \varphi_m )^n = \lim_{m \rightarrow \infty}  e^{\varphi_m} \Omega_m  =   \Omega_{KE}.$$

This concludes the proof of the proposition as well as Theorem B.1.

\qed

\end{proof}

\begin{remark}

The existence of such a canonical metric does not depend on the finite generation of the canonical ring of $X$.  The regularity and uniqueness of such K\"aher-Einstein metrics will be investigated in our future study.

\end{remark}

\bigskip


\subsection{Algebraic manifolds of positive Kodaira dimension}

Let $X$ be an $n$-dimensional nonsingular algebraic variety of Kodaira dimension $\kappa$, where $0< \kappa < n$. Let $\Phi^{\dagger}: X^{\dagger} \rightarrow Y^{\dagger}$ be the Iitaka fibration of $X$ unique up to birational equivalence. Let $\Phi_m$ be the pluricanonical map associated to the linear system $|m K_X|$. Then for $m$ sufficiently large there exists a commutative diagram

\begin{equation}
\begin{diagram}
\node{X} \arrow{s,l,..}{{\small \Phi_m} }     \node{X^{\dagger}}  \arrow{w,t}{\pi^{\dagger}} \arrow{s,r}{{\small \Phi^{\dagger}}}\\
\node{Y_m}      \node{Y^{\dagger}} \arrow{w,b,..}{\mu_m}
\end{diagram}
\end{equation}
as in Section 2.2. A very general fibre of $\Phi^{\dagger}$ has Kodaira dimension zero.

We will generalize the notion of the Weil-Petersson metric on a special local
deformation space of K\"ahler manifolds of zero Kodaira dimension.

\begin{definition}\label{wp2}

Let $f: \mathcal{X}\rightarrow B$ be a holomorphic nonsingular fibration over a ball $B\in \mathbf{C}^{\kappa}$ such that for any $t\in B$, $\mathcal{X}_t = f^{-1} (t)$ is a nonsingular fibre of dimension $n-\kappa$.
Let $t=(t_1, ..., t_{\kappa})$ be the holomorphic coordinates of $B$, where $\kappa=\dim {\mathcal M}$.
Then each
$\frac{\partial}{\partial t_i}$ corresponds to an element
$\iota(\frac{\partial}{\partial t_i})\in H^1(\mathcal{X}_t,
T_{\mathcal{X}_t})$ through the Kodaira-Spencer map $\iota$. We assume that there exists a nontrivial holomorphic
$(n-\kappa,0)$-form $\Psi$ on $\mathcal{X}$ such that its restriction on each fibre $\mathcal{X}_t$ is also a nontrivial holomorphic
$(n-\kappa,0)$-form on $\mathcal{X}$.  Then the
Weil-Petersson metric is defined by the $L^2$-inner product
\begin{equation}
\left( \frac{\partial}{\partial t_i}, \frac{\partial}{\partial
\bar{t_j}} \right)_{\omega_{WP}}=\frac{\int_{\mathcal{X}_t}\Psi\lrcorner\iota(\frac{\partial}{\partial
t_i})\wedge\overline{\Psi\lrcorner\iota(\frac{\partial}{\partial
t_i})}}{\int_{\mathcal{X}_t}\Psi\wedge\overline{\Psi}},
\end{equation}
where $\Psi\lrcorner\iota(\frac{\partial}{\partial t_i})$ is the
contraction of $\Psi$ and $\iota(\frac{\partial}{\partial t_i})$.
\end{definition}

The metric defined above is a pseudometric and the associated closed $(1,1)$-form $\omega_{WP}$ is only semi-positive in general. In Definition \ref{wp2}, $\omega_{WP}$ depends on the choice of the holomorphic $(n-\kappa, 0)$- form $\Psi$, however, it is uniquely determined for the Iitaka fibration $\Phi^{\dagger} : X^{\dagger} \rightarrow Y^{\dagger}$ by the following lemma.

\begin{lemma} The Weil-Petersson metric is well-defined on a Zariski open  set of $Y^{\dagger}$ and it is unique.
\end{lemma}

\begin{proof} Let $B\subset Y^{\dagger}$ be a nonsingular open neighborhood such that each fibre over $B$ is nonsingular. Let $\mathcal{X} = (\Phi^{\dagger})^{-1} (B) \subset X^{\dagger}$. Without loss of generality, we assume  $H^0(X^{\dagger}, K_{X^{\dagger}})\neq \phi$, otherwise we can replace $K_X$ by a sufficiently large power of $K_X$,  and so $H^0(B, (\Phi^{\dagger})_*
\Omega^{n-\kappa}_{\mathcal{X}/B}) \neq \phi$. Then the assumption in Definition \ref{wp2} can be satisfied. Let $\Psi_1$ and $\Psi_2$ be two holomorphic $(n-\kappa)$-forms over $B$ in Definition \ref{wp2}. Since a very general fibre has Kodaira dimension zero, for a general point $t\in B$,

$$ \left. \frac{\Psi_1}{\Psi_2}\right|_{\mathcal{X}_t} = constant$$
and so $ \left. \frac{\Psi_1}{\Psi_2}\right|_{\mathcal{X}_t} $ is constant on each point $t\in B$ since  $ \left. \frac{\Psi_1}{\Psi_2}\right|_{\mathcal{X}_t}$ is smooth on $\mathcal{X}$.

One can also represent $\omega_{WP}$ as the curvature form of the
first Hodge bundle $f_*
\Omega^{n-\kappa}_{\mathcal{X}/B}$ with the same assumption as in Definition \ref{wp2}. Let $\Psi$ be a
nonzero local holomorphic section of $f_*
\Omega^{n-\kappa}_{\mathcal{X}/B}$ and one can define
the hermitian metric $h_{WP}$ on $f_*
\Omega^{n-\kappa}_{\mathcal{X}/B}$ by
\begin{equation}
|\Psi_t|^2_{h_{WP}}=\int_{\mathcal{X}_t}\Psi_t\wedge\overline{\Psi_t}.
\end{equation}
Then the Weil-Petersson metric is given by
\begin{equation}
\omega_{WP}=\ric(h_{WP}).
\end{equation}Therefore the Weil-Petersson metric is unique on $B$.

\qed
\end{proof}

We choose a sequence of resolution of indeterminacies of the pluricanonical systems $\Phi_{m!}=
\Phi_{|m! K_X|}$

\begin{equation}
\begin{diagram}
\node{X}      \node{X_{m_0}}  \arrow{w,t}{\pi_{m_0}}  \node{X_{m_0+1}}  \arrow{w,t}{\pi_{m_0+1}} \node{ \cdot\cdot\cdot} \arrow{w,t}{\pi_{m_0+2}} \node{X_m} \arrow{w,t}{\pi_m}  \node{X_{m+1}} \arrow{w,t}{\pi_{m+1}} \node{\cdot\cdot\cdot} \arrow{w,t}{\pi_{m+2}}
\end{diagram}
\end{equation}

for $m_0$ sufficiently large, such that

\begin{enumerate}

\item

$$(\bar{\pi}_m)^* ( m! K_X ) = L_m + E_m,$$
where $\bar{\pi}_m= \pi_m\circ\pi_{m-1} \circ...\circ\pi_{m_0}$.

\item

$$E_m = \sum_{j} c_{m, j} E_{m, j} $$ is the fixed part of $\left| (\bar{\pi}_m) ^* ( m!K_X ) \right|$ with each $E_{m, j}$ being a divisor with simple normal crossings.

\item $L_m$ is a globally generated line bundle on $X_m$.

\end{enumerate}

Let $Y_m$ be the variety determined by the pluricanonical system $| m! K_X |$ and $\Psi_m  = \Psi_{|L_m|}$ be the rational map associated to the linear system $|L_m|$. Then we have the following diagram

\begin{equation}
\begin{diagram}\label{diag5}
\node{X} \arrow{se,b,..}{\Phi_m}      \node[2]{X_m} \arrow[2]{w,t}{\bar{\pi}_m} \arrow{sw,r}{\Psi_m} \\
\node[2]{Y_m}
\end{diagram}
\end{equation}

There exists a commutative diagram by choosing $m_0$ sufficiently large

\begin{equation}\label{diagone}
\begin{diagram}
\node{X}      \node{X_{m_0}}  \arrow{w,t}{\pi_{m_0}} \arrow{s,r}{\Psi_{m_0}} \node{X_{m_0+1}}  \arrow{w,t}{\pi_{m_0+1}} \arrow{s,r}{\Psi_{m_0+1}} \node{ \cdot\cdot\cdot} \arrow{w,t}{\pi_{m_0+2}} \node{X_m} \arrow{w,t}{\pi_m} \arrow{s,r}{\Psi_m} \node{X_{m+1}} \arrow{w,t}{\pi_{m+1}} \arrow{s,r}{\Psi_{m+1}} \node{\cdot\cdot\cdot} \arrow{w,t}{\pi_{m+2}}\\
\node{}       \node{Y_{m_0} }   \node{ Y_{m_0+1}} \arrow{w,t,..}{\mu_{m_0+1}}     \node{ \cdot\cdot\cdot} \arrow{w,t,..}{\mu_{m_0+2}} \node{Y_{m}} \arrow{w,t,..}{\mu_m}  \node{Y_{m+1}} \arrow{w,t,..}{\mu_{m+1}}  \node{\cdot\cdot\cdot} \arrow{w,t,..}{\mu_{m+2}}
\end{diagram}
\end{equation}
of rational maps and holomorphic maps where the horizontal maps are birational and  $\mu_m $ is given by the projection from $|m!K_X|$ to $|(m-1)!K_X |$ as a subspace of $|m!K_X|$.

Let $\{ \sigma_{m, j}\}_{j=0}^{d_m}$ be a basis of $H^0(X, m!K_X)$ and $\{\zeta_{m, j}\}_{j=0}^{d_m}$ be a basis of $H^0(X, L_m)$ such that

$$(\bar{\pi}_m)^* \sigma_{m, j} = \zeta_{m, j} E_m  \in H^0( X_m, (\bar{\pi}_m)^* (m! K_X)).$$
Then
$$\frac{ \left( \sum_{j=0}^{d_{m+1}} \sigma_{m+1, j}\otimes \overline{\sigma_{m+1, j}} \right)^{\frac{1}{(m+1)!}}} { \left( \sum_{j=0}^{d_m} \sigma_{m, j}\otimes \overline{\sigma_{m, j}} \right)^{\frac{1}{m!}}} <\infty,$$
where $ 1+d_m = \dim H^0(X, m!K_X)$.

Let $\Omega_m = \left( \sum_{j=0}^{d_m} \sigma_{m, j}\otimes \overline{\sigma_{m, j}} \right)^{\frac{1}{m!}}$ and $\Omega_{(m)}= (\Psi_m) _* \Omega_m$. Let $$\omega_{m}= \frac{1}{m} \ddbar\log \left( \sum_{j=0}^{d_m} |\zeta_{m, j}|^2\right)$$ be the normalized Fubini-Study metric on $Y_{m}$.
Then the same argument in Proposition \ref{Fso} gives the following Lemma.

\begin{lemma}  There exists $p=p(m)>1$ such that

\begin{equation}
F_{m}=\frac{\Omega_{(m)}}{\omega_{m}^{\kappa}} \in L^p \left( Y_{m}, \omega_{m}^{\kappa} \right).
\end{equation}

\end{lemma}

The following proposition is immediate as in Section 3.2.
\begin{proposition} There exists a unique solution $\varphi_{m} \in PSH( Y_m, \omega_m) \cap C^0( Y_{m} )$ to the following Monge-Amp\`ere equation

\begin{equation}
(\omega_{m} + \ddbar \varphi_{m} )^{\kappa} = e^{\varphi_{m}} \Omega_{(m)}.
\end{equation}
Furthermore, $\varphi_{m}$ is smooth on a Zariski open set of $Y_{m}$.

\end{proposition}

For simplicity, we abuse the notations and use $\Omega_m$ for $ (\bar{\pi}_{k})^* \Omega_m$ for $k\geq m_0 $ without causing confusion.

\begin{lemma}\label{constantfibre} There exists a Zariski open set $U$ of $Y_{m+1}$ such that$\frac{\Omega_m }{\Omega_{m+1}}$ is constant along each fibre of $\Psi_{m+1}$ over $U$.
\end{lemma}

\begin{proof} Let $F =\frac{  \Omega_m }{\Omega_{m+1}}$ and it is easy to see that $F$ is smooth. We consider the following diagram

\[
\begin{diagram}
\node{X_{m+1}} \arrow{s,l,}{{\small \Psi_{m+1} } }    \node{X^{\dagger}}  \arrow{w,t,..}{f} \arrow{s,r}{{\small \Phi^{\dagger}}}\\
\node{Y_{m+1}}      \node{Y^{\dagger}} \arrow{w,b,..}{g }
\end{diagram}
\]
where  $\Phi^{\dagger}: X^{\dagger} \rightarrow Y^{\dagger}$ is an Iitaka fibration of $X$.

A very general fiber of $\Phi_{\dagger}$ is nonsingular of Kodaira dimension $0$. Let $F_{s_0}= (\Phi^{\dagger})^{-1}(s_0)$ be a very general fibre. Consider $B=\{ s \in Y^{\dagger}~|~ |s-s_0| < \delta \}$ such that for any $s\in B$, $F_{s}= (\Phi^{\dagger})^{-1}(s)$ is non-singular. Let $\eta$ be a nowhere-vanishing holomorphic $\kappa$-form on $B$. Then

$$\left. \frac{(\pi^{\dagger})^*\sigma_{m, j}}{ \eta^m} \right|_{s_0} \in H^0(F_{s_0}, m!K_{F_{s_0}})$$
and
$$\left. \frac{ (\pi^{\dagger})^*\sigma_{m+1, j}}{ \eta^{m+1}} \right|_{s_0} \in H^0(F_{s_0}, (m+1)!K_{F_{s_0}}).$$

Since $\dim PH^0( F_{s_0}, k K_{F_{s_0}})=0$ for any $k\geq1$ and $\frac{  \Omega_m }{\Omega_{m+1}}< \infty$, $\frac{ \Omega_m }{\Omega_{m+1}}$ must be constant on each $F_{s_0}$.

Therefore $f^* F$ is constant on a very general fibre of $\Phi^{\dagger}$. By Hartog's theorem, $f^* F$ is smooth on $X^{\dagger}$ and so $f^*F$ is the pullback of a function on a Zariski open set of $Y^{\dagger}$. By the commutative diagram, on a Zariski open set of $X_{m+1}$, $F$ is the pullback of function on $Y_{m+1}$ and so $F$ has to be constant on a very general fibre of $\Psi_{m+1}$.

\qed
\end{proof}

Hence $\frac{\Omega_{m}}{\Omega_{m+1}}$ descends to a function on $Y_{m+1}$. By the commutative diagram, the following lemma is immediate.
\begin{lemma} \label{sss} For each $m \geq m_0$, $$(\psi_{m+1})_*  \Omega_m = (\mu_{m+1})^* \Omega_{(m)}.$$

\end{lemma}

The following corollary is immediate by the diagram \ref{diagone}, Lemma \ref{constantfibre} and \ref{sss}.

\begin{corollary} For each $m \geq m_0$, on $X_{m+1}$

$$ \frac{\Omega_m}{\Omega_{m+1} }  = \frac { (\mu_{m+1})^* \Omega_{(m)}}{\Omega_{(m+1)}}.$$

\end{corollary}

\begin{lemma} Let $U_{m+1}=\{ s \in Y_{m+1}~|~ \frac{\Omega_{m+1}}{  \Omega_m } < \infty \}$. Then $\mu_{m+1}$ is holomorphic on $U_{m+1}$.

\end{lemma}

\begin{proof} On $(\Psi_{m+1})^{-1} \left( U_{m+1} \right)$,

$$ 0< \frac{ \left(\sum_j |\sigma_{m, j}|^2\right)^{\frac{1}{m} }} { \left( \sum_j |\sigma_{m+1, j} |^2 \right)^{\frac{1}{m+1}}} < \infty. $$
Both $L_{m+1}$ and $(\pi_m)^* L_{m}$ are globally generated. Therefore the base locus of
$$\left\{\frac{ \left((\bar{\pi}_{m+1})^*\sigma_{m, j}\right)^{m+1} } {E_{m+1}} \right\}_{j=0}^{d_m}$$
 is outside $U_{m+1}$ and so $\mu_{m+1}$ is well defined on $U_{m+1}$.

\qed
\end{proof}

\begin{proposition}

For any $m\geq m_0$, the measure $e^{\varphi_{(m)}} \Omega_{(m)}$ is increasing, that is, on $Y_{m+1}$

\begin{equation}
(\mu_{m+1})^* \left( e^{\varphi_{(m)}} \Omega_{(m)} \right)  \leq e^{\varphi_{(m+1)}} \Omega_{(m+1)}.
\end{equation}

\end{proposition}

\begin{proof} By resolution of singularities, we can assume $Y_{(m+1)}$ is non-singular by replacing $Y_{m+1}$ by its non-singular model. Let $\omega_{m}' = (\mu_{m+1})^* \omega_{m}$, $\varphi_{m}'= (\mu_{m+1})^* \varphi_{m}$ and $\Omega'_{(m)} = (\mu_{m+1})^*\Omega_{(m)}$. So $\varphi'_{m}$ satisfies the following Monge-Amp\`ere equation
$$ \left( \omega'_{m} + \ddbar \varphi'_{m} \right) ^{\kappa} = e^{\varphi'_{m}} \Omega'_{(m)}$$
on $U_{m+1}$, where the $\omega'_m$ is a smooth and positive, and so the Monge-Amp\`ere mass $\left( \omega'_{m} + \ddbar \varphi'_{m} \right) ^{\kappa}$ is well-defined.
Also $\varphi_{m+1}$ is the solution of

$$  \left( \omega_{m+1} + \ddbar \varphi_{m+1} \right) ^{\kappa} = e^{\varphi_{m+1}} \Omega_{(m+1)}.$$

Define $$\psi= \varphi_{m+1} -\varphi'_{m} + \log \frac{\Omega_{(m+1)}}{\Omega'_{(m)}}$$ and $V=\{ z \in Y_m ~|~ \psi \leq 0 \}$. It is easy to see that

$$V \subset U_{m+1}.$$
since both $\varphi_{m+1}$ and $\varphi'_{m}$ are in $L^{\infty}(Y_{m+1})$.
In particular,  $\varphi'_{m}\in C^0( U_{m+1} )$.

On $U_{m+1}$, $\log \left(\frac{ |E_{m+1}|^{2m} }{ |(\pi_m)^*E_m |^{2(m+1)}}\right)$ is smooth and so

\begin{eqnarray*}
&&\omega_{m+1}+\ddbar\varphi_{m+1} \\
&=& \omega'_{m}+ \ddbar \varphi'_{m} + \ddbar \log  (\varphi_{m+1} -\varphi'_{m}) + \ddbar \log \left( \frac{\left( \sum_j |\zeta_{m+1, j} |^2 \right)^{\frac{1}{m+1}}} {\left( \sum_j |(\pi_m)^*\zeta_{m, j} |^2 \right)^{\frac{1}{m}}}\right)\\
&=& \omega'_{m}+ \ddbar \varphi'_{m} + \ddbar \psi - \ddbar \log \left(\frac{ |E_{m+1} |^{\frac{2}{m+1}} }{ |(\pi_m)^*E_m|^{\frac{2}{m}}}\right)\\
&=& \omega'_{m}+ \ddbar \varphi'_{m}+ \ddbar \psi .
\end{eqnarray*}

Hence $\psi$ satisfies the following equation on $U_{m+1}$

\begin{equation}
\frac{ \left( \omega'_{m}+ \ddbar \varphi'_{m} + \ddbar \psi \right)^{\kappa} }
 { \left( \omega'_{m}+ \ddbar \varphi'_{m} \right)^{\kappa} }= e^{\psi}.
\end{equation}

By the comparison test, we have on $V=\{ z \in Y_{m+1}~|~ \psi(z) < 0 \} \subset U_{m+1}$

\begin{eqnarray*}
&&\int_V \left( \omega'_{m} + \ddbar \varphi'_{m} \right)^{\kappa}\\
&\leq& \int_V \left( \omega'_{m} + \ddbar \varphi'_{m} + \ddbar \psi \right)^{\kappa}\\
&=&\int_V e^{\psi} \left( \omega'_{m} + \ddbar \varphi'_{m} \right)^{\kappa}.
\end{eqnarray*}

Therefore $\psi=0$ on $V$ due to the fact that it is continuous on $V$, and so on $Y_{m+1}$

$$\psi\geq 0.$$

This completes the proof.

\qed
\end{proof}

 For simplicity we again use $\varphi_m$ for $(\Psi_m)^{*} \varphi_m$. We can consider $\varphi_m$ as a function on $X$, $X_m$ or $X^{\dagger}$ since they  are all birationally equivalent.

\begin{corollary} \label{mon2}
For any $m\geq m_0$, the measure $e^{\varphi_{m}} \Omega_{m}$ is increasing, that is,

\begin{equation}
 e^{\varphi_m} \Omega_{m}   \leq e^{\varphi_{m+1}} \Omega_{m+1}.
\end{equation}
\end{corollary}

\begin{lemma} $e^{\varphi_m} \Omega_m$ is continuous on $X$ and let $h_m = e^{-\varphi_m} \Omega_m^{-1}$ be the hermitian metric on $K_X$. Then $$\Theta_{h_m}\geq 0.$$

\end{lemma}

\begin{proof} Obviously, $e^{\varphi_m}\Omega_m$ is continuous on $X\backslash Bs(|m!K_X|)$ since the inverse of $\bar{\pi}_m$ is isomorphic from $X_m\backslash (\bar{\pi}_m)^{-1} (Bs(|m!K_X|))$ to $X \backslash  (Bs(|m!K_X|)$. On the other hand, $\varphi_m \in L^{\infty}(X)$ and $\Omega_m$ vanishes exactly on $Bs(|m!K_X|)$. Hence $e^{\varphi_m} \Omega_m$ is continuous.

$\ddbar \log \Omega_m + \ddbar \varphi_m \geq 0$ on $X \backslash  Bs(|m!K_X|)$ and $Bs(|m!K_X|)$ is a closed complete pluripolar set of $X$. By the fact that $e^{\varphi_m}\Omega_m$ is bounded above on $X$, $$\Theta_{h_m} = \ddbar \log \Omega_m + \ddbar \varphi_m $$
extends to a closed positive current on $X$ by local argument. This completes the proof of the lemma.

\qed
\end{proof}

\begin{lemma} \label{upbound2} Let $\Omega$ be a smooth volume form on $X^{\dagger}$. Then there exist a divisor $D$ of $X^{\dagger}$ and a constant $C>0$ such that

\begin{equation}
   e^{\varphi_{m}} \Omega_{m}   \leq  |S_D|_{h_D}^{-2}\Omega,
\end{equation}
where $S_D$ is a defining section of $[D]$ and $h_D$ is a fixed smooth hermitian metric of the line bundle associated to $[D]$.

\end{lemma}

\begin{proof} We consider again the following commutative diagram

\[
\begin{diagram}
\node{X_{m}} \arrow{s,l,}{{\small \Psi_m} }     \node{X^{\dagger}}  \arrow{w,t,..}{f_m} \arrow{s,r}{{\small \Phi^{\dagger}}}\\
\node{Y_{m}}      \node{Y^{\dagger}} \arrow{w,b,..}{g_m }
\end{diagram}
\]

Let $\omega$ be a K\"ahler metric and $\Omega_0$ be a smooth and nowhere vanishing volume form on $X^{\dagger}$. From the commutative diagram, on a Zariski open set of $X^{\dagger}$ we have

$$( \ddbar \log \Omega_m + \ddbar \varphi_m )^{\kappa} \wedge \omega^{ n- \kappa} = e^{\varphi_m} (\Phi^{\dagger})^*(\Phi^{\dagger})_* \Omega_m \wedge \omega^{n-\kappa}.$$
Let $\psi_m = \varphi_m + \log\frac{\Omega_m}{\Omega}$, $\chi_0= \ddbar\log \Omega_0$.

\medskip

$\psi_m$ satisfies the following Monge-Amp\`ere equation on $X^{\dagger} \backslash (\pi^{\dagger})^{-1}Bs(|m_0! K_X|)$,

$$ ( \chi_0 + \ddbar \psi_m )^{\kappa} \wedge \omega^{ n- \kappa} = F e^{\psi_m} \Omega_0,$$
where $$F= \frac{(\Phi^{\dagger})^*(\Phi^{\dagger})_* \Omega_m \wedge \omega^{n-\kappa}}{\Omega_m} =  \frac{(\Phi^{\dagger})^*(\Phi^{\dagger})_* \Omega_{m_0} \wedge \omega^{n-\kappa}}{\Omega_{m_0}}.$$

Let $D_1$ be a divisor of $X^{\dagger}$ containing $(\pi^{\dagger})^{-1}Bs(|m_0! K_X|)$ such that the defining section $S_1$ of $[D_1]$  satisfies

$$ |S_1|^2_{h_1} \frac{ \omega^n }{ F \Omega_0} < \infty, $$
where $h_1$ is a fixed smooth hermitian metric of the line bundle associated to $[D_1]$.
Obviously, $D_1$ is independent of the choice of $m$. Let $\Theta_{h_1} = -\ddbar \log h_1$.

Let $D_2$ be an ample divisor of $X^{\dagger}$ independent on the choice of $m$ such that for a smooth hermitian metric $h_2$, we have

$$ \Theta_{h_2} =- \ddbar\log h_2 > - \Theta_{h_1}.$$
Let $S_2$ be the defining section of $D_2$

Let $D_3$ be a divisor of $X^{\dagger}$ depending on the choice of $m$ such that $\psi_m \in C^{\infty}(X^{\dagger} \backslash D_3)$. Let $S_3$ be the defining section of $D_3$ and $h_3$ be a fixed smooth hermitian metric of the line bundle associated to $[D_3]$.
We define $\Theta_{h_3} = -\ddbar\log h_3$ and for $\epsilon>0$  sufficiently small

$$ \Theta_{h_2} > - \Theta_{h_1} - \epsilon \Theta_{h_3}.$$
Now we let
$$\psi_{m, \epsilon} = \varphi_m + \log |S_1|^2_{h_1} +\log |S_2|^2_{h_2} + \epsilon \log |S_3|^2_{h_3},$$ and so $\tilde{\varphi}_{m, \epsilon} $ satisfies

\begin{equation}
( \chi_0 + \Theta_{h_1} +  \Theta_{h_2}+ \epsilon \Theta_{h_3} +  \ddbar\log \psi_{m, \epsilon} )^{\kappa} \wedge \omega^{n-\kappa} =  \frac{F e^{\psi_{m, \epsilon} }}{ |S_1|^2_{h_2}|S_2|^2_{h_2}(|S_3|^2_{h_3})^{\epsilon} }  \Omega_0.
\end{equation}

The maximum of $\psi_{m, \epsilon} $ can only be achieved in $X^{\dagger}\backslash ( D_1 \cup D_2 \cup D_3)$. Then by the maximum principle,

\begin{equation}
\sup_{X^{\dagger}} \psi_{m, \epsilon} \leq  \sup_{X^{\dagger}} \left( |S_1|^2_{h_1}|S_2|^2_{h_2}(|S_3|^2_{h_3})^{\epsilon} \frac{ ( \chi_0+ \Theta_{h_1}+\Theta_{h_2} + \epsilon \Theta_{h_3}  )^{\kappa} \wedge \omega^{n-\kappa} } {F \Omega_0}  \right) = C_{m, \epsilon},
\end{equation}
where $\lim_{ \epsilon \rightarrow \infty} C_{m, \epsilon} = C_{m, 0}= \sup_{X^{\dagger}} \left( |S_1|^2_{h_1}|S_2|^2_{h_2} \frac{ ( \chi_0+ \Theta_{h_1}+\Theta_{h_2}   )^{\kappa} \wedge \omega^{n-\kappa} } {F \Omega_0}  \right)$ and $C_{m, 0}$ is independent of the choice of $m$.

Now let $\epsilon$ tend to $0$. Then there exists a constant $C>0$ independent of the choice of $m$ such that

$$\sup_{X^{\dagger}} \psi_{m} \leq C,$$
that is, there exists $C'>0$ independent of $m$ such that
$$ e^{\varphi_m} \Omega_m \leq C' |S_1|^{-2}_{h_1}|S_2|^{-2}_{h_2} \Omega_0.$$

\qed
\end{proof}

\begin{proposition} \label{conv1} There exists a measure $\Omega_{can}$ on $X$ such that

\begin{enumerate}

\item

$\Omega_{can} = \lim_{m\rightarrow \infty} e^{\varphi_m} \Omega_m.$

\item $(K_X, \Omega_{can}^{-1})$ is an analytic Zariski decomposition. Furthermore,

$$\frac{\Omega_{can}} {\Omega_0} < \infty,  ~~~~~{\it and } ~~~~~ \frac{\Psi_{X, \epsilon}}{\Omega_{can}} < \infty.$$

\end{enumerate}

\end{proposition}

\begin{proof} By Corollary \ref{mon2} and Lemma\ref{upbound2}, we can define $\Omega_{can}$

$$\Omega_{can} = \lim_{m\rightarrow \infty} e^{\varphi_m} \Omega_m$$
and $\frac{\Omega_m}{\Omega_0} <\infty$.   Since $\varphi_m + \log \frac{\Omega_m}{\Omega_0} \in PSH(X, \chi_0) $ and $\{\varphi_m + \log \frac{\Omega_m}{\Omega_0}\}_{m=m_0}^{\infty}$ is convergent in $L^1(X)$. Also $PSH(X, \chi_0)\cap L^1(X)$ is closed in $L^1(X)$. Therefore

$$ \lim_{m\rightarrow \infty}  \left( \varphi_m + \log \frac{\Omega_m}{\Omega_0} \right) = \log \frac {\Omega_{can}}{\Omega_0}$$
in $PSH(X, \chi_0)\cap L^1(X)$ and $\log \frac {\Omega_{can}}{\Omega_0} <\infty$.

Let $h_{can} = \Omega_{can}^{-1}$ be the hermitian metric on $K_X$. By the construction of $\Omega_m$,

$$ |\sigma|^2_{h_{can}^m} < \infty$$
for any section $\sigma\in H^0(X, mK_X)$. Hence $\frac{\Psi_{X, \epsilon}} {\Omega_{can}} < \infty$ and

$$H^0(X, \OO_X(mK_X)\otimes \II(h_{can}^m))\rightarrow H^0(X, \OO_X(mK_X))$$
is an isomorphism and $(K_X, h_{can})$ is an analytic Zariski decomposition.

\qed

\end{proof}

\begin{proposition}\label{conv2}

Let $\Omega^{\dagger} = (\pi^{\dagger})^* \Omega_{can}$. There exists a closed positive $(1,1)$-current $\omega^{\dagger}$   on $Y^{\dagger}$ such that $\left(\Phi^{\dagger}\right)^*\omega^{\dagger}= \ddbar \log \Omega^{\dagger}$ on a Zariski open set of $X^{\dagger}$. Furthermore, on a Zariski open set of $Y^{\dagger}$, we have

    \begin{equation}
    (\omega^{\dagger})^{\kappa} = (\Phi^{\dagger})_* \Omega^{\dagger},
    \end{equation}
and so
    \begin{equation}
    \ric( \omega^{\dagger} )=- \omega^{\dagger} +\omega_{WP}.
    \end{equation}

\end{proposition}

\begin{proof}

Let $\psi_m = \varphi_m  + \log \frac{\Omega_{m}}{\Omega_{m_0}}$. Both $\psi_m$ and $\varphi_m$ descend to $Y^{\dagger}$ and by Proposition \ref{conv1} $$\lim_{m \rightarrow \infty} \psi_m = \psi_{\infty}=\log \frac{\Omega^{\dagger}}{\Omega_{m_0}}.$$
Consider
\[
\begin{diagram}
\node{X_{m}} \arrow{s,l,}{{\small \Psi_m} }     \node{X^{\dagger}}  \arrow{w,t,..}{f_m} \arrow{s,r}{{\small \Phi^{\dagger}}}\\
\node{Y_{m}}      \node{Y^{\dagger}} \arrow{w,b,..}{g_m }
\end{diagram}
\]
For simplicity, we use $\omega_{m}$ for $(g_m)^*\omega_m$.
Let $D_{m_0}$ be a divisor of $Y^{\dagger}$ such that on $X\backslash D_{m_0}$,  $\omega_0$ is  smooth  and $\log \frac{\Omega^{\dagger}}{\Omega_{m_0}} < \infty$. Also on $Y\backslash D_{m_0}$,

$$ (\omega_{m_0} + \ddbar \psi_m)^{\kappa} = e^{\varphi_m} (\Phi^{\dagger})_* \Omega_m.$$
Since $\psi_m $ converges uniformly on any compact set of $Y^{\dagger}\backslash D_{m_0}$ to $\psi_{\infty}$, we have on $Y^{\dagger} \backslash D_{m_0} $,
$$ (\omega_{m_0} + \ddbar \psi_\infty )^{\kappa}= \lim_{m\rightarrow \infty} ( \omega_{m_0} + \ddbar \log \psi_m )^n = \lim_{m\rightarrow \infty} e^{\varphi_m} (\Phi^{\dagger})_* \Omega_m = (\Phi^{\dagger})_* \Omega^{\dagger} .$$
Let $\omega^{\dagger} = \omega_{m_0} + \ddbar \psi_\infty$. Since it is a closed positive current on $Y^{\dagger}\backslash D_{m_0}$ and it can be extended to a closed positive current on $Y$.

Also $(\Phi^{\dagger})^*\omega_{m_0} = \ddbar \log \Omega_{m_0}$ on $X^{\dagger} \backslash ( \pi^{\dagger})^{-1}(Bs(|m_0! K_X|))$. This implies that  on $X^{\dagger} \backslash ( \pi^{\dagger})^{-1}(Bs(|m_0! K_X|))$,

$$ (\Phi^{\dagger})^* \omega^{\dagger} = \ddbar \log \Omega^{\dagger}.$$

Furthermore, we have $\frac{\Omega_{m_0}}{\Omega^{\dagger} } = \frac {(\Phi^{\dagger})_*\Omega_{m_0}}{ (\Phi^{\dagger})_*\Omega^{\dagger} }$ and so

\begin{eqnarray*}
\ric(\omega^{\dagger}) &=& \ddbar\log (\Phi^{\dagger})_*\Omega^{\dagger} \\
&=& - \omega^{\dagger} + \ddbar\log (\Phi^{\dagger})_*\Omega_{m_0}- \ddbar\log \Omega_{m_0} \\
&=&-\omega^{\dagger} + \omega_{WP}.
\end{eqnarray*}

\qed
\end{proof}

Proposition \ref{conv1} and Proposition \ref{conv2} conclude the proof of Theorem B.2.


\subsection{Uniqueness assuming finite generation of the canonical ring}

If the canonical ring $(X, R_X)$ is finitely generated, the canonical model $X_{can}$  is unique and can be constructed by the pluricanonical system $|mK_X|$ for sufficiently large $m$.
In this section, we will prove the uniqueness of the canonical measures constructed in Section 4.3 and 4.4 by assuming finite generation of the canonical ring. Furthermore, the canonical measure can be considered as a birational invariant.

\begin{theorem}\label{finitegentype} Let $X$ be an algebraic manifold of general type. If the canonical ring $R(X, K_X)$ is finitely generated, the K\"ahler-Einstein measure in Theorem B.1 is constructed in finite steps. Furthermore, it is continuous on $X$ and smooth on a Zariski open dense set of $X$.
\end{theorem}

Theorem \ref{finitegentype} is an immediate consequence from the proof of Theorem B.1 with the assumption of finite generation of the canonical ring. The following theorem is proved in \cite{EyGuZe1}.

\begin{theorem}\label{EyGuZe1} Let $X$ be an algebraic manifold of general type. If the canonical ring $R(X, K_X)$ is finitely generated, $X_{can}$ will have only canonical singularities and there exists a unique K\"ahler-Einstein metric $\omega_{can}$ on $X_{can}$ with a continuous potential.

\end{theorem}

Theorem C.1 is then proved as a corollary of Theorem \ref{finitegentype} and Theorem \ref{EyGuZe1}.

\begin{corollary}Let $X$ be an $n$-dimensional algebraic manifold of general type. Suppose that the canonical ring $R(X, K_X)$ is finitely generated and $\pi: X\dasharrow X_{can}$ is the pluricanonical map. Let $\omega_{can}$ be the unique K\"ahler-Einstein metric on $X_{can}$ as in Theorem \ref{EyGuZe1} and $\Omega_{KE}=\pi^* (\omega_{can}^n)$. Then $(X, \Omega_{KE}^{-1})$  coincides with the  analytic Zariski decomposition constructed in Theorem B.1.

\end{corollary}

\begin{proof} Since the canonical ring $R(X, K_X)$ is finitely generated, the pluricanonical map is stabilized for sufficiently large power so that the proof of Theorem B.1 terminates in finite steps. It is then straightforward to check that the K\"ahler-Einstein metric constructed in Theorem B.1 satisfies the same Monge-Amp\`ere equation on the unique canonical model of $X$ in Theorem \ref{EyGuZe1} (see \cite{EyGuZe1}).

\end{proof}

We shall now prove Theorem C.2.

\begin{definition}\label{push2} Suppose that $X$ is an $n$-dimensional algebraic manifold of Kodaira dimension $0<\kappa<n$ and the canonical ring $R(X, K_X)$ is finitely generated. Let  $\Phi: X \dasharrow X_{can}$ be the pluricanonical map. There exists a nonsingular model $X^{\dagger}$ of $X$ and the following diagram holds

\begin{equation}
\begin{diagram}\label{diag6}
\node{X} \arrow{se,b,..}{\Phi}      \node[2]{X^{\dagger}} \arrow[2]{w,t}{\pi^{\dagger}} \arrow{sw,r}{\Phi^{\dagger}} \\
\node[2]{X_{can}}
\end{diagram}
\end{equation}
where $\pi^{\dagger}$ is  barational and the generic fibre of $\Phi^{\dagger}$ has Kodaira dimension $0$.

\begin{enumerate}

\item

Then the pushforward measure $\Phi_* \Omega$ on $X_{can}$ is defined by

\begin{equation}
\Phi_* \Omega = (\Phi^{\dagger})_* \left(  (\pi^{\dagger})^* \Omega \right).
\end{equation}

\item

Let $\Phi=\Phi_m$ be the pluricanonical map associated to a basis $\{ \sigma_{j_m} \}_{j_m=0}^{d_m}$ of the linear system $\left| m K_X \right|$, for $m$ sufficiently large. Let $\Omega_m = \left( \sum_{j_m=0}^{d_m} \sigma_{j_m}\otimes \overline{\sigma_{j_m}} \right)^{\frac{1}{m}}$ and $\omega_{FS}$ be the Fubini-Study metric of $\mathbf{CP}^{d_m}$ restricted on $X_{can}$ associated to $\Phi_m$. Then we defined $\bar{\omega}_{WP}$ by

\begin{equation}
\bar{\omega}_{WP} = \frac{1}{m} \omega_{FS} + \ddbar \log \Phi_* \Omega_m.
\end{equation}

\end{enumerate}

In particular, $\bar{\omega}_{WP}$ coincides with $\omega_{WP}$ in Definition \ref{wp2} on a Zariski open set of $X_{can}$.

\end{definition}

\begin{lemma}\label{pushlemma} $\Phi_* \Omega$ is independent of the choice of the diagram in Definition \ref{push2}.

\end{lemma}

\begin{proof} Let $\rho$ be a test function on $X_{can}$. Then

$$\int_{X_{can}} \rho \Phi_*\Omega = \int_{X^{\dagger}} \left( (\Phi^{\dagger})^* \rho \right)  (\phi^{\dagger})^* \Omega
= \int_X  \left( \Phi^* \rho \right)\Omega,$$
which is independent of the choice of the diagram \ref{diag6}.

\qed
\end{proof}

Since the generic fibre of $\Phi^{\dagger}$ has Kodaira dimension $0$, by the same argument in Lemma \ref{constantfibre}, we have the following lemma.

\begin{lemma} Let  $\{ \sigma^{(1)}_{j_p} \}_{j_p=0}^{d_p}$ and $\{ \sigma^{(2)}_{j_q} \}_{j_q=0}^{d_q}$ be basis of the linear systems $\left| p K_X \right|$ and $\left| q K_X \right|$, for $p$ and $q$ sufficiently large. Let $\Omega^{(1)} = \left( \sum_{j_p=0}^{d_p} \sigma_{j_p}\otimes \overline{\sigma_{j_p}} \right)^{\frac{1}{p}}$ and $\Omega^{(2)} = \left( \sum_{j_q=0}^{d_q} \sigma_{j_q}\otimes \overline{\sigma_{j_q}} \right)^{\frac{1}{q}}$.
Then $(\pi^{\dagger})^* \left( \frac{ \Omega^{(1)} } {\Omega^{(2)}} \right)$ is constant on any generic fibre and so

\begin{equation}
(\pi^{\dagger})^* \left( \frac{ \Omega^{(1)} } {\Omega^{(2)}} \right)  = (\Phi^{\dagger})^*\left(\frac{ \Phi_* \Omega^{(1)}} {\Phi_* \Omega^{(2)}  } \right).
\end{equation}

\end{lemma}

\begin{lemma} The definition of $\bar{\omega}_{WP}$ only depends on $X$.

\end{lemma}

\begin{proof} By Lemma \ref{pushlemma}, the definition of $\Omega_{WP}$ does not depend on the choice of the diagram \ref{diag5}. Let  $\{ \sigma^{(1)}_{j_p} \}_{j_p=0}^{d_p}$ and $\{ \sigma^{(2)}_{j_q} \}_{j_q=0}^{d_q}$ be basis of the linear systems $\left| p K_X \right|$ and $\left| q K_X \right|$, for $p$ and $q$ sufficiently large. Let $\Omega^{(1)} = \left( \sum_{j_p=0}^{d_p} \sigma_{j_p}\otimes \overline{\sigma_{j_p}} \right)^{\frac{1}{p}}$ and $\Omega^{(2)} = \left( \sum_{j_q=0}^{d_q} \sigma_{j_q}\otimes \overline{\sigma_{j_q}} \right)^{\frac{1}{q}}$.

Let $\omega_{FS}^{(1)}$ and $\omega_{FS}^{(2)}$ be the Fubini-Study metrics of $\mathbf{CP}^{d_p}$ and $\mathbf{CP}^{d_q}$ restricted on $X_{can}$ associated to $\Phi_p$ and $\Phi_q$. Then by avoiding the base locus of $R(X^{\dagger}, K_{X^{\dagger}})$, there exist  a Zariski open set $U$ of $X_{can}$ and a Zariski open set $V$ of $X^{\dagger}$, such that on $V$

$$\frac{1}{p}( \Phi^{\dagger})^* \omega_{FS}^{(1)} = \ddbar \log (\phi^{\dagger})^* \Omega^{(1)}, ~~~~~~ \frac{1}{q}(\Phi^{\dagger})^* \omega_{FS}^{(2)} = \ddbar \log (\phi^{\dagger})^*\Omega^{(2)}$$
and so on $U$,

\begin{equation}\label{ext}
\frac{1}{p} \omega_{FS}^{(1)}- \frac{1}{q}\omega_{FS}^{(2)} = \ddbar \log \left(\frac{ \Phi_* \Omega^{(1)}} {\Phi_* \Omega^{(2)}  } \right).
\end{equation}

Since $\frac{1}{p} \omega_{FS}^{(1)}$ and $ \frac{1}{q}\omega_{FS}^{(2)}$ are in the same class, and $\log \left(\frac{ \Phi_* \Omega^{(1)}} {\Phi_* \Omega^{(2)}  } \right)$ is in $L^{\infty}(X_{can})$, Equation \ref{ext} holds everywhere on $X_{can}$. Therefore the following equality completes the proof of the lemma

$$\frac{1}{p} \omega_{FS}^{(1)}- \ddbar \log \Phi_* \Omega^{(1)} = \frac{1}{q}\omega_{FS}^{(2)}- \ddbar \log \Phi_* \Omega^{(2)}.$$

\qed
\end{proof}

\begin{theorem} Suppose that $X$ is an $n$-dimensional algebraic manifold of Kodaira dimension $0<\kappa<n$. If the canonical ring $R(X, K_X)$ is finitely generated, $X_{can}$ is the canonical model of $X$, then there exists a unique canonical measure $\Omega_{can}$ on $X$ satisfying

\begin{enumerate}

\item $\Omega_{can}$ is continuous on $X$ and  smooth on a Zariski open set of $X$.

\item  ${\displaystyle 0< \frac{\Omega_{can}}{ \Psi_{X,M} } < \infty}$ and $(K_X, \Omega_{can}^{-1})$ is an analytic Zariski decomposition.

\item Let  $\Phi: X \dasharrow X_{can}$ be the pluricanonical map. Then there exists a unique closed positive $(1,1)$-current  $\omega_{can}$ with bounded local potential on $X_{can}$ such that  $\Phi^* \omega_{can}= \ddbar\log \Omega_{can}$ outside the base locus of the pluricanonical system. Furthermore,
$$ ( \omega_{can} )^{\kappa} = \Phi_* \Omega_{can},$$
so on a Zariski open set of $X_{can}$ we have
$$\ric(\omega_{can})= -\omega_{can} + \bar{\omega}_{WP}.$$

\end{enumerate}

Furthermore, $\Omega_{can}$ is invariant under birational transformations.

\end{theorem}

\begin{proof}

 If $R(X, K_X)$ is finitely generated, there exists the following diagram

\[
\begin{diagram}
\node{X} \arrow{se,b,..}{\Phi}      \node[2]{X^{\dagger}} \arrow[2]{w,t}{\pi^{\dagger}} \arrow{sw,r}{\Phi^{\dagger}} \\
\node[2]{X_{can}}     \\
\end{diagram}
\]
where $X_{can}$ is the canonical model of $X$ and $X^{\dagger}$ is the resolution of  the stable base locus of the pluricanonical systems such that $\left( \pi^{\dagger}\right)^* MK_X = L_M+ E_M$ for sufficiently large $M$, where $L_M$ is globally generated and $E_M$ is the fixed part of $\left| \left( \pi^{\dagger}\right)^* MK_X \right|$ with $E_M$ being a divisor with normal crossings. $X^{\dagger}$ is an Iitaka fibration over $X_{can}$ such that the generic fibre has Kodaira dimension $0$. Let $\{ \sigma_{j }\}_{j=0}^{d_M}$ be a basis of $H^0(X, M K_X)$ and $\{\zeta_{j}\}_{j=0}^{d_M}$ be a basis of $H^0(X, L_M)$ such that

$$\left(\pi^{\dagger} \right)^*  \sigma_{j} = \zeta_{j} E_M.$$

Let $\Omega= \pi^{\dagger}\left(\sum_{m=0}^M \sum_{j_m=0}^{d_m} |\sigma_{j_m}|^{\frac{2}{m}}\right)$ be a degenerate smooth volume form on $X^{\dagger}$ and $\omega= \frac{1}{M} \ddbar \log \left(\sum_{j=0}^{d_M} |\zeta_j|^2\right)$. Then the following Monge-Amp\`ere equation has a unique continuous solution $\varphi$ on $X_{can}$

\begin{equation}
\left( \omega + \ddbar \varphi \right)^{\kappa} = e^{\varphi} (\Phi^{\dagger})_* \Omega.
\end{equation}
Furthermore, $\varphi$ is smooth on a Zariski open set $X_{can}^{\circ}$ of $X_{can}$ and so is $\omega_{can}=\omega+\ddbar \varphi$.

Let $\Theta = \frac{\Omega} { ( \Phi^{\dagger} )^* ( \Phi^{\dagger} )_* \Omega}$ be an $(n-\kappa, n-\kappa)$-current  on $X^{\dagger}$. On a generic fibre $F$, $\left. \Theta \right|_F= \eta\wedge\bar{\eta}$ for some $\eta\in H^0(X, K_{F})$. So without loss of generality, we assume that on $X_{can}^{\circ}$,

$$\omega_{WP} = \ddbar \log \Theta - \frac{1}{M} \ddbar\log |E_M|^2.$$

Therefore on $X_{can}^{\circ}$

$$\ric(\omega_{can}) = -\omega_{can} + \bar{\omega}_{WP}.$$
On the other hand, we define

$$\Omega_{can} = ( (\pi^{\dagger})^{-1})^* \left( e^{\varphi} \Omega \right) =  e^{\Phi^*\varphi} \left(\sum_{m=0}^M \sum_{j_m=0}^{d_m} |\sigma_{j_m}|^{\frac{2}{m}}\right).$$
From the regularity of $\varphi$, $\Omega_{can}$ is continuous on $X$ and smooth on a Zariski open set of $X$ and $$\ddbar \log \Omega_{can}= \ddbar \log \left(\sum_{m=0}^M \sum_{j_m=0}^{d_m} |\sigma_{j_m}|^{\frac{2}{m}}\right) + \ddbar \Phi^*\varphi = \Phi^* \omega_{can}.$$

We then shall prove the uniqueness of $\Omega_{can}$. Suppose there exists another measure $\Omega'$ satisfying the assumptions in the theorem. Then let $\Omega'= e^{\varphi '} \Omega$. Since $\Phi^* \omega_{FS}= \ddbar\log \Omega$ and $\ddbar\log \Omega' -\ddbar\log \Omega$ is a pullback from $X_{can}$. Therefore on a Zariski open set  a generic fibre $F$ of $\Phi^{\dagger}$,  we have $$\sqrt{-1} \partial_F \dbar_F \log (\pi^{\dagger})^*\left(\frac{\Omega'} {\Omega}=0\right) = \sqrt{-1} \partial_F \dbar_F (\pi^{\dagger})^* \varphi'.$$

Since $(\pi^{\dagger})^* \varphi'\in L^{\infty}(F)$, $(\pi^{\dagger})^* \varphi' $ is constant along $F$. So $\varphi'$ descends to $X_{can}$ and satisfies the following Monge-Amp\`ere equation

\begin{equation}\label{uniqeqn}
 (\omega+ \ddbar \varphi')^{\kappa} = e^{\varphi'} \Phi_* \Omega.
\end{equation}
By the uniqueness of the solution of Equation \ref{uniqeqn}, $\Omega' =\Omega_{can}$ and we have proved the uniqueness of $\Omega_{can}$.

Finally we shall prove that $\Omega_{can}$ and $\Phi^* \omega_{can}$ are birational invariants.
Suppose $X_{(1)}$ and $X_{(2)}$ are birational with $X_{can}$ being the canonical model. Then we have the following diagram
\[
\begin{diagram}
\node{X_{(1)}} \arrow{se,b,..}{\Phi^{(1)}}      \node[2]{X_{(2)}} \arrow[2]{w,t,..}{f} \arrow{sw,r,..}{\Phi^{(2)}} \\
\node[2]{X_{can}}
\end{diagram}
\]
where $f$ is birational,
 $\Phi^{(1)}$ and $\Phi^{(2)}$ are the pluricanonical maps. Fix $\Omega$ on $X_{(1)}$ as constructed as in the proof of uniqueness, then by Hartog's theorem, $f^* \Omega $ is smooth and can be constructed the same way. By replacing $X_{(1)}$ and $X_{(2)}$ by their Iitaka fibration, it is straightforward to show that $$(\Phi^{(1)})_* \Omega = (\Phi^{(2)})_* \left( f^* \Omega \right).$$

 Let $(\Phi^{(1)})^* \omega= \ddbar\log \Omega$,  $\Omega_{(1)}= e^{\varphi_{(1)}}\Omega$ and  $\Omega_{(2)}= e^{\varphi_{(2)}} f^*\Omega$ be the unique canonical measures on $X_{(1)}$ and $X_{(2)}$. Both $\varphi_{(1)}$ and $\varphi_{(2)}$ descend to $PSH(X, \omega)\cap L^{\infty}(X_{can})$ and satisfy

$$(\omega+\ddbar\varphi)^{\kappa} = e^{\varphi} (\Phi^{(1)}) _* \Omega.$$
The uniqueness of the solution to the above Monge-Amp\`ere equation implies that $\varphi_{(1)} = \varphi_{(2)}$ and so

$$ f^* \Omega_{(1)} = \Omega_{(2)}.$$

\qed
\end{proof}


\section{The K\"ahler-Ricci flow}

\subsection{Reduction of the normalized K\"ahler-Ricci flow}

Let $X$ be an $n$-dimensional compact K\"ahler manifold. A
K\"ahler metric can be given by its K\"ahler form $\omega$ on $X$.
In local coordinates $z_1, ..., z_n$, we can write $\omega$ as
$$\omega=\sqrt{-1}\sum_{i, j=1}^n
g_{i\bar{j}}dz_i\wedge dz_{\bar{j}},$$ where $\{g_{i\bar{j}}\}$ is
a positive definite hermitian matrix function. Consider the normalized
K\"ahler-Ricci flow
\begin{equation}\label{krflow2}
\left\{
\begin{array}{rcl}
&&{ \displaystyle \frac{\partial \omega(t,\cdot)}{\partial t} = -
 \ric(\omega(t,\cdot))- \omega(t,\cdot),}\\
&&\\
&& \omega(0,\cdot)=\omega_0 .
\end{array} \right.
\end{equation}
where
$ \ric(\omega(t,\cdot))$ denotes the Ricci curvature of
$\omega(t,\cdot)$ and $\omega_0$ is a given K\"ahler metrics.

Let $Ka(X)$ denote the K\"ahler cone of $X$, that is,
$$Ka(X)=\{ [\omega]\in H^{1,1}(X, \mathbf{R})~|~
[\omega]>0\}.$$
%
Suppose that $\omega(t,\cdot)$ is a solution of (\ref{krflow2}) on
$[0,T)$. Then its induced equation on K\"ahler classes in $Ka(X)$
is given by the following ordinary differential equation
\begin{equation}
\left\{
\begin{array}{rcl}
&&{\displaystyle \ddt{[\omega]}}  =  \displaystyle{
-2\pi c_1(X) - [\omega]}\\
&&\\
&&[\omega]|_{t=0} = [ \omega_0] .
\end{array} \right.
\end{equation}
It follows
$$[\omega(t,\cdot)]=-2\pi c_1(X)+e^{-t}([\omega_0]+2\pi c_1(X)).$$
Now if we assume that  canonical bundle
$K_X$ is semi-positive, then for a sufficiently large integer $m$, the pluricanonical map associated to
$H^0(X, m K_X)$ gives rise to an algebraic fibre space $f: X \rightarrow X_{can}$, where $X_{can}$ is the canonical model of $X$.  Recall the Kodaira dimension $kod(X)$ of $X$ is defined to
be the dimension of $X_{can}$.
Moreover, there is a smooth K\"ahler
form $\chi$ as the Fubini-Study metric associated to a basis of $H^0(X, mK_X)$ on the normal K\"ahler space $X_{can}$ such that $f^*\chi$ represents $-2\pi
c_1(X)$. Choose the reference K\"ahler metric $\omega_t$ by

\begin{equation}\label{ref metric}
\omega_t=\chi+e^{-t}(\omega_0-\chi).
\end{equation}
Here we abuse the notation by identifying $\chi$ and $f^*\chi$ for
simplicity. Then the solution of (\ref{krflow2}) can be written as
$$\omega=\omega_t+\sqrt{-1}\partial\dbar \varphi.$$
We can always choose a smooth volume form $\Omega$ on $X$ such
that $\ric(\Omega)=\chi$. Then the evolution for the K\"ahler
potential $\varphi$ is given by the following initial value
problem:
\begin{equation} \label{**}
\left\{
\begin{array}{rcl}
&& {\displaystyle \ddt{\varphi}=
\log\frac{e^{(n- \kappa)t}(\omega_t+\sqrt{-1}\partial\dbar\varphi)^n}{\Omega
}-
\varphi }\\
&&\\
&&\varphi|_{t=0} = 0,  \end{array} \right.
\end{equation}
where $\kappa= kod (X)$.


\subsection{K\"ahler-Ricci flow on algebraic manifolds with semi-positive canonical line bundle}

\begin{theorem}\label{main2}

Let $X$ be a nonsingular algebraic variety with semi ample
canonical line bundle $K_X$ and so $X$ admits a
holomorphic fibration over its canonical model $X_{can}$ $f:
X\rightarrow X_{can}$.  Then for any initial K\"ahler metric, the
K\"ahler-Ricci flow (\ref{krflow1}) has a global solution
$\omega(t,\cdot)$ for all time $t\in [0, \infty)$ satisfying:

\begin{enumerate}

\item $\omega(t,\cdot)$ converges to $f^*\omega_{\infty}\in -2\pi
c_1(X)$ as currents for a positive current $\omega_{\infty}$ on
$\Sigma$.

\item $\omega_\infty$ is smooth on $X_{can}^{\circ}$ and satisfies
the generalized K\"ahler-Einstein equation on $X_{can}^{\circ}$
\begin{equation}
\ric(\omega_{\infty})=-\omega_{\infty}+\omega_{WP},
\end{equation}
where $\omega_{WP}$ is the induced Weil-Petersson metric.

\item for any compact subset $K\in X_{reg}$, there is a constant
$C_K$ such that
\begin{equation}
||R(t, \cdot)||_{L^\infty(K)}+ e^{(n- \kappa) t}
\sup_{f^{-1}(s)\in K} ||(\omega(t,\cdot))^{n-\kappa} |_{f^{-1}(s)}||_{L^\infty} \le C_K
.
\end{equation}

\end{enumerate}

\end{theorem}

\begin{corollary}\label{main3}

Let $X$ be a nonsingular algebraic variety with semi-ample
canonical bundle. If $X_{can}^{\circ}=X_{can}$, i.e., $X_{can}$ is
nonsingular and $f: X\rightarrow X_{can}$ has no singular fibres,
then for any initial K\"ahler metric, the K\"ahler-Ricci flow
(\ref{krflow1}) converges to a smooth limit metric
$f^*\omega_{\infty}\in K_X$ satisfying

\begin{equation}
\ric(\omega_{\infty})=-\omega_{\infty}+\omega_{WP}.
\end{equation}

\end{corollary}

\noindent{\it Step 1.} {\bf Zeroth order and volume estimates}

\bigskip\noindent We will  first derive the zeroth order estimates for $\varphi$ and
$\frac{d\varphi}{dt}$.

\begin{lemma} \label{0th}
Let $\varphi$ be a solution of the K\"ahler-Ricci flow (\ref{**}).
There exists a  constant $C>0$  such that on $[0, \infty) \times X$

\begin{enumerate}

\item    $\varphi\leq C$,

\item    $\ddt{\varphi}\leq C$,

\item    $\displaystyle \frac{e^{(n-\kappa)t} \omega^n}{\Omega} \leq C$.

\end{enumerate}

\end{lemma}

\begin{proof} The lemma is a straightforward application of the maximum principle and can be proved by the same argument as in \cite{SoTi}.

\qed
\end{proof}

\begin{proposition} \label{c0dm}
There exists a constant $C>0$  such that on $[0, \infty) \times X$
\begin{equation}
|\varphi|\leq C.
\end{equation}
\end{proposition}
\begin{proof} Rewrite the parabolic flow as a family of Monge-Amp\`ere equations

$$ \displaystyle{  (\omega_t + \ddbar \varphi )^n  =  e^{\ddt{\varphi} + \varphi -(n-\kappa)t } \Omega.}$$

We will apply Theorem \ref{D-P} by letting $F(t, \cdot)=e^{\ddt{\varphi} + \varphi -(n-\kappa)t }$. Notice that there exists a constant $C_1>0$ such that $ 0 < F \leq C_1  e^{-(n-\kappa)t} $,
$$ \frac{ e^{-(n-\kappa)t} }{C_1} \leq [\omega_t]^n \leq C_1 e^{-(n-\kappa)t} $$ and
$$ \frac{ e^{-(n-\kappa)t} }{C_1} \chi^{\kappa}\wedge \omega_0^{n-\kappa} \leq \omega_t^n \leq C_1 e^{-(n-\kappa)t} \Omega.$$
The assumptions in Theorem \ref{D-P} for $F$ and $\omega_t$ are satisfied. Therefore $\sup_X \varphi - \inf_X \varphi$ is uniformly bounded for all $t\in [0, \infty)$. Since $\varphi$ is uniformly bounded from above, the proposition is proved and the uniform $C^0$-estimate is obtained.

\qed

\end{proof}

The following estimate can be proved in the same way as in \cite{SoTi}.

\begin{lemma}\label{time} There exists a divisor $D$ on $X_{can}$ and  constants $C_1$ and $C_2>0$ such that

\begin{equation}
\ddt{\varphi}\geq  C_1 \log|S|_{h}^2 - C_2,
\end{equation}
where $S$ is a defining section of $f^* D$ and $h$ is a fixed smooth hermitian metric of the line bundle associated to $[f^*D]$.
\end{lemma}

\noindent{\it Step 2.} {\bf Partial second order estimates and collapsing}

\begin{proposition} \label{partialvol}
There  exist a divisor $D$ on $X_{can}$ and constants $\lambda$, $C>0$ such that
\begin{equation}
\tr_{\omega}(\chi)\leq \frac{C}{|S|_h^{2\lambda}},
\end{equation}
where $S$ is a defining section of $f^*D$ and $h$ is smooth hermitian metric of the line bundle associated to the divisor $[f^* D]$.
\end{proposition}

\begin{proof} Since $X_{can}$ might be singular, we can consider the nonsingular model $f': X' \rightarrow Y'$ for $f: X\rightarrow X_{can}$ such that following diagram commutes

\[
\begin{diagram}
\node{X} \arrow{s,l}{{\small f} }     \node{X'}  \arrow{w,t}{\pi} \arrow{s,r}{{\small f'}}\\
\node{X_{can}}      \node{Y^{'}} \arrow{w,b}{\mu }
\end{diagram}
\]
where $\pi$ and $\mu$ are birational.

Let $\varphi'= \pi^* \varphi$, $\chi'=\mu^* \chi$, $\omega_0'= \pi^* \omega$ and $\omega'=\pi^* \omega$. We also write $\chi'$ for $(f' )^*\chi'$ for simplicity.

Let $ \theta $ be a K\"ahler form on $Y'$ such that $\theta \geq \chi'$. For simplicity, we identify $\theta$ and $\chi'$ with $(f')^*\theta$ and $(f')^*\chi'$. Since $\chi'$ is semi-positive induced by the Fubini-Study metric and it only vanishes along a subvariety of $Y'$ with finite order, there exists a divisor $D_1$ on $Y'$ such that

$$\theta \leq \frac{1}{|S_1|^2_{h_1}} \chi',$$
where $S_1$ is a defining section of $(f')^* D_1$ with $h_1$ a smooth hermitian metric of the line bundle associated to $[(f')^* D_1]$. Without loss of generality, we can assume that the support of $\mu^* D$ in Lemma \ref{time} is contained in $D_1$.

Let $$u=\tr_{g'} (\theta) = (g')^{i\overline{j}}
\theta_{i \overline{j}}, $$ where $g'$ is the K\"ahler metric associated to $\omega'$. The K\"ahler-Ricci flow for $\omega$ can be pulled back to the K\"ahler-Ricci flow for $\omega'$ on $X'$ outside the exceptional divisors. Let $\Delta'$ be the Laplace operator associated to $g'$. We have then

$$ \tr_{g'} (\theta) \leq \frac{1}{|S_1|^2_{h_1}} \tr_{g'} ( \chi').$$

Following the similar calculation in \cite{SoTi}, we
have
\begin{equation}
 \left( \ddt{}-\Delta' \right) \log u
 \leq
 C(u+1)
\end{equation}

Since $[\chi']$ is big and semi-ample, there exists a divisor $D_2$ on $Y'$ such that $[\chi'] -\epsilon [D_2]$ is ample for any $\epsilon > 0$. Then let $S_2$ be the defining section for $(f')^* D_2$ and there exists a smooth hermitian metric $h_2$ on the line bundle associated to $[D_2]$ such that
$$\chi' - \epsilon\Theta_{h_2}= \chi' +  \epsilon \ddbar \log h_2 >0.$$ For simplicity, we identify $h_2$ and $(f')^* h_2$.

Let $D_3$ be a divisor on $X'$ containing the exceptional divisor of $\pi$ on $X'$. Let $S_3$ be the defining section of  $[D_3]$.  There exists a smooth hermitian metric $h_3$ on the line bundle associated to $[D_3]$ such that for all sufficiently small $\delta >0$ $$ \omega_0' + \delta \Theta_{h_3} \geq 0.$$

We define
$$\varphi'_{\epsilon} = \varphi' - \epsilon \log |S_2|^2_{h_2}  .$$
Then there exists a constant $C>0$ depending on $\epsilon$ such that

$$\Delta' \varphi'_{\epsilon} = n - \tr_{g'}(\omega_t' - \epsilon \Theta_{h_2} ) \leq n - C \tr_{g'}\theta - e^{-t} \tr_{g'} (\omega_0') = n - Cu - e^{-t} \tr_{g'} (\omega_0').$$
Calculate for sufficiently large $A>0$ and small $\epsilon>0$ and $\delta>0$
\begin{eqnarray*}
&& \left( \ddt{}-\Delta' \right) \left( \log \left( |S_1|^4_{h_1} |S_3|^{2\delta e^{-t}}_{h_3}u \right) - 2A \varphi'_{\epsilon}\right) \\
%
%
&\leq& -Au + C -\delta e^{-t} \log|S_3|^2_{h_3} - 2A \ddt \varphi'.
\end{eqnarray*}
for all $t>0$ in $X'\backslash ( D_1 \cup D_2 \cup D_3)$.

The maximum of $\log \left( |S_1|^4_{h_1} |S_3|^{2\delta e^{-t}}_{h_3}u \right) - 2A \varphi'_{\epsilon}$ can only be achieved on $X'\backslash ( D_1 \cup D_2 \cup D_3)$.
The maximum principle implies that there exist constants $\lambda$, $C>0 $ independent of $\delta$ such that for all $(t, z) \in [0, \infty)\times X'$

$$ u(t, z) \leq  C \left( |S_1|^{-2\lambda}_{h_1} |S_2|^{-2\lambda}_{h_2} |S_3|^{-2\delta e^{-t}}_{h_3} \right) (z).$$

The proposition is then proved by  letting $\delta \rightarrow 0$.

\qed
\end{proof}




%
%


From the uniform upper bound of $\omega^n$, one immediately concludes that the volume of a regular fibre of $f$ tends to $0$ exponentially fast uniformly away from the singular fibres.

\begin{corollary} \label{partialreg} There exists a divisor $D$ on $X_{can}$ and constants $\lambda$, $C>0$ such that for all $t\geq 0$ and $s\in X_{can}$

\begin{equation}
 \frac{ \left( \left. \omega\right|_{X_s} \right)^{n-\kappa} } { \left( \left. \omega_0 \right|_{X_s} \right)^{n-\kappa} } \leq e^{ -(n-\kappa)t} \frac{C} {|S|_h^2} ,
\end{equation}
where $\left. \omega_0 \right|_{X_s} $ and $\left. \omega\right|_{X_s}$ are the restriction of $\omega_0$ and $\omega_s$ on $X_s = f^{-1} (s)$, $S$ is a defining section of $f^*D$ and $h$ is smooth hermitian metric of the line bundle associated to the divisor $[f^* D]$.

\end{corollary}

\begin{proof} Notice that

$$\frac{ \left( \left. \omega\right|_{X_s} \right)^{n-\kappa} } { \left( \left. \omega_0 \right|_{X_s} \right)^{n-\kappa} } = \frac{ \omega^{n-\kappa} \wedge \chi^{\kappa}} {\omega_0^{n-\kappa}\wedge \chi^{\kappa}} = \frac{ \omega^{n-\kappa} \wedge \chi^{\kappa}} {\omega^n}  \frac{ \omega^n} {\omega_0^{n-\kappa}\wedge \chi^{\kappa}} \leq C \left( \frac{ \omega^{n-1} \wedge \chi} {\omega^n} \right)^{\kappa} \frac{ \omega^n} {\omega_0^{n-\kappa}\wedge \chi^{\kappa}} .$$

The corollary is then proved by Lemma \ref{0th} and Proposition \ref{partialvol}.

\qed
\end{proof}

\begin{corollary}
For any compact set $K \subset X_{can}^{\circ}$, there exists a constant $C_K$ such that for all $t\geq 0$ and $s\in K$

\begin{equation}
 \sup_{X_s} \varphi(t, \cdot) - \inf_{X_s} \varphi(t, \cdot) \leq C_K e^{ -t} .
\end{equation}

\end{corollary}

\begin{proof} The Poincare and Sobolev constants with respect to $\left. \omega_0 \right|_{X_s}$ for $s\in K$ are uniformly bounded. The proof of the corollary is achieved by Corollary \ref{partialreg} and Moser's iteration in Yau's $C^0$-estimate for the Calabi conjecture.

\qed
\end{proof}

%

\noindent{\it Step 3.} {\bf Gradient estimates}
The gradient estimates in this section are obtained in the same way as in \cite{SoTi} and it is an adaption from the gradient estimate in \cite{ChYa} and the argument in \cite{Pe} to obtain a uniform bound for $\left |\nabla \ddt{\varphi}\right |_g$ and the
scalar curvature $R$.  Let
$u=\ddt{\varphi}+\varphi=\log\frac{e^{(n-\kappa)t}\omega^2}{\Omega}$. The
evolution equation for $u$ is given by
\begin{equation}
\ddt{u}=\Delta u+\tr_{\omega}(\chi)-(n-\kappa).
\end{equation}
We will obtain a gradient estimate for $u$, which will help us
bound the scalar curvature from below. Note that $u$ is uniformly
bounded from above, so we can find a constant $A>0$ such that
$A-u\geq 1$.

\begin{theorem}\label{gradient estimate}
There exist constants  $\lambda$,  $C>0$ such that
\begin{enumerate}
\item   $|S|_h^{2\lambda}|\nabla u|^2\leq
C(A-u),$
\item   $ -|S|_h^{2\lambda}\Delta u\leq
C(A-u),$
\end{enumerate}
where $\nabla$ is the gradient operator with respect to the metric
$g$ associated to $\omega$ along the flow and $|\cdot|=|\cdot|_g$.
\end{theorem}

Theorem \ref{gradient estimate} is proved the same way as in \cite{SoTi} with little modification.
The following corollary is immediate by Theorem \ref{gradient estimate}, Lemma \ref{c0dm} and  Lemma \ref{time}.

\begin{corollary}\label{gradient2} For any $\delta>0$, there exist constants $\lambda$,
$C >0$ such that

\begin{enumerate}

\item $|S|_h^{2\lambda}|\nabla u|^2\leq C,$

\item $ -|S|_h^{2\lambda}\Delta u\leq C.$

\end{enumerate}
\end{corollary}


Now we are in the position to prove a uniform bound for the scalar
curvature. The following corollary tells that the K\"ahler-Ricci flow
will collapse with bounded scalar curvature away from the singular
fibres.

\begin{corollary}\label{gradient3}
Along the K\"aher-Ricci flow (\ref{krflow1}) the scalar curvature
$R$ is uniformly bounded on any compact subset of $X_{reg}$. More
precisely, there exist constants $\lambda$,  $C>0$ such that
\begin{equation}
  -C \leq R \leq \frac{C}{|S|_h^{2\lambda}}.
\end{equation}
\end{corollary}

\begin{proof} It suffices to give an upper bound for $R$ since the scalar curvature $R$ is  uniformly bounded from below by the maximum principle (cf. \cite{SoTi}).
Notice that
$R_{i\overline{j}}=-u_{i\overline{j}}-\chi_{i\overline{j}}$ and
then
$$R=-\Delta u-\tr_{\omega}(\chi).$$
By Corollary \ref{gradient2} and the partial second order estimate, there exist constants $\lambda_6$,  $C>0$ such that
$$R\leq \frac{C}{|S|_h^{2\lambda}}.$$  \qed

\end{proof}

\bigskip

\noindent{\it Step 4.} {\bf Uniform convergence}

\bigskip \noindent Let $\varphi_{\infty} $ be the unique solution solving equation (\ref{MAeqn}) in Theorem
\ref{MAthm}. We identify $f^*\varphi_{\infty}$ and $\varphi_{\infty}$ for convenience.

Since $K_X$ is semi-ample, there exists an ample line bundle $L$ on $X_{can}$ such that $K_X = f*L = \frac{1}{m} f^*\mathcal{O}(1)$ for a fixed pluricanonical map. Let $D$ be an ample divisor on $X_{can}$ such that $[D] = \mu  [L]$  for a sufficiently large integer $\mu $,  $X_{can} \setminus X_{can}^{\circ} \subset D$ and $\varphi_{\infty}\in C^{\infty}(X_{can}\setminus D) $. Let $S_D$ be the defining section of $D$. Let $h_{FS}$ be the Fubini study metric on $\mathcal{O}(1)$ induced by the pluricanonical map. Then there exists a continuous hermitian metric $h_D = \left( h_{FS} \right) ^{ \frac{ \mu }{ m } } e^{ - \mu \varphi_{\infty}}$  on $ L^{\mu} $ such that $ - \ddbar \log h_D = \mu \chi_{\infty}$ since $\varphi_{\infty}$ is continuous.

We define $$ B_{r} (D) = \{  y \in X_{can}  ~| ~ dist_{\chi} ( y, D ) \leq r \} $$ be the geodesic  tubular neighborhood of $D$ with respect to $\chi$ and we let $\mathcal{B}_r(D) = f^{-1} \left( B_{r}(D) \right)$.

Since  $\varphi_{\infty}$ is bounded
on $X$ and $\varphi$ is uniformly bounded from above. Therefore for any $\epsilon>0$, there exists
$r_{\epsilon}>0$ with $\lim_{\epsilon\rightarrow
0}r_{\epsilon}=0$, such that for any $z\in \mathcal{B}_{r_{\epsilon}}(D)$ and $t \geq 0$ we have
$$ \left( \varphi-\varphi_{\infty}+\epsilon\log|S_D|_{h_D}^2 \right)(t, z)<-1$$ and
$$ \left( \varphi-\varphi_{\infty}-\epsilon\log|S_D|_{h_D}^2 \right)(t, z)>1$$

Let $\eta_{ \epsilon}$ be a smooth cut off function on $X_{can}$ such that $\eta_{\epsilon} = 1$ on $X_{cam}  \setminus  B_{r_{\epsilon}} ( D ) $ and $\eta_{\epsilon} = 0$ on $ B _{\frac{r_{\epsilon}}{2}}(D)$.

 Suppose the semi-flat closed form is given by
$\omega_{SF}=\omega_0+\sqrt{-1}\partial\dbar\rho_{SF}$ and
$\rho_{SF}$ blows up near the singular fibres. We let $\rho_{\epsilon}$ be an approximation for $\rho_{SF}$ given by
$$\rho_{\epsilon}= \left( f^* \eta_{\epsilon} \right)\rho_{SF}.$$
We also define
$\omega_{SF,\epsilon}=\omega_0+\sqrt{-1}\partial\dbar\rho_{\epsilon}$.
Now we define the twisted difference  of $\varphi$ and
$\varphi_{\infty}$ by

$$\psi^{-}_{\epsilon}=\varphi-(1+\epsilon) \varphi_{\infty}- e^{-t} \rho_{\epsilon}
+\epsilon\log|S_D|^2_{h_D} $$
and
$$\psi^{+}_{\epsilon}=\varphi-(1-\epsilon)\varphi_{\infty}- e^{-t} \rho_{\epsilon}
-\epsilon\log|S_D|^2_{h_D} $$
where $S_D$ is the defining section of $D$ and $h$ is a fixed smooth hermitian metric of the line bundle induced by $[D]$. We identify $f^* \left( |S_D|^2_{h_D} \right)$ and $|S_D|^2_{h_D} $ for convenience.

\begin{proposition}\label{conv}   There exists $\epsilon_0>0$ such that for any
$0<\epsilon<\epsilon_0$, there exists $T_{\epsilon}>0$ such that
for any $z\in X$ and $t>T_{\epsilon}$ we have

\begin{equation}
\psi^{-}_{\epsilon}(t, z)\leq 2 \epsilon.
\end{equation}
and
\begin{equation}
\psi^{+}_{\epsilon}(t, z)\geq - 2 \epsilon.
\end{equation}
\end{proposition}

\begin{proof}  The evolution equation for $\psi_{\epsilon}^{-}$ is given by

\begin{equation}
\ddt{\psi^{-}_{\epsilon}}=\log
\frac{e^{(n-\kappa)t}((1+\epsilon\mu - e^{-t})\chi_{\infty}+e^{-t}\omega_{SF,
\epsilon}+\sqrt{-1}\partial\dbar\psi^{-}_{\epsilon})^n}{C_{n, \kappa}  \chi_{\infty}^{\kappa} \wedge
\omega_{SF}^{n-\kappa}}-\psi^{-}_{\epsilon} - \epsilon \varphi_{\infty} + \epsilon\log|S|^2_h.
\end{equation}
Since $\rho_\epsilon$ is bounded on $X$, we can always choose
$T_1>0$ sufficiently large such that for $t>T_1$

\begin{enumerate}

\item
$\psi^{-}_{\epsilon}(t, z)<-\frac{1}{2}$ on
$\mathcal{B}_{r_{\epsilon}}(D)$,

\item
$\sum_{p=0}^{\kappa-1} C_{n, p} e^{-(n-p)t} \left|\frac{\chi_{\infty}^p\wedge\omega_{SF}^{n-p}}{\chi_{\infty}^{\kappa}\wedge\omega_{SF}^{n-\kappa}}\right|\leq
\epsilon$ on $X\setminus \mathcal{B}_{r_{\epsilon}}(D)$.

\end{enumerate}

We will discuss in two cases for $t>T_1$.

\begin{enumerate}

\item If $\psi^{-}_{\epsilon, max}(t)=\max_{
X}\psi^{-}_{\epsilon}(t, \cdot)=\psi^{-}_{\epsilon}(t,
z_{max,t})>0$ for all $t>T_1$. Then $z_{max,t}\in X\setminus
\mathcal{B}_{r_{\epsilon}}( D )$ for all $t>T_1$ and so $\omega_{SF, \epsilon}(z_{max, t}) = \omega_{SF} (z_{max, t})$. Applying
the maximum principle at $z_{max,t}$,  we have
\begin{eqnarray*}
&&\ddt{\psi^{-}_{\epsilon}}(t, z_{max,t})\\
\\
&\leq&
\left( \log\frac{e^{(n-\kappa)t}( (1+\epsilon\mu - e^{-t}) \chi_{\infty} - \epsilon \varphi_{\infty} + e^{-t}\omega_{SF,\epsilon})^n}
{C_{n, \kappa} \chi_{\infty}^{\kappa}\wedge\omega_{SF}^{n-\kappa}}-\psi^{-}_{\epsilon} - \epsilon \varphi_{\infty} +\epsilon\log|S|^2_h\right)
(t,z_{max,t})\\
&&\\
&=&\left( \log\frac{ \sum_{p=0}^{\kappa} \left(\begin{array}{c}
                                                                n \\ \kappa
                                                              \end{array}\right)(1+\epsilon\mu -e^{-t})^{p}\chi_{\infty}^p\wedge\omega_{SF,\epsilon}^{n-p}}{  \left(\begin{array}{c}
                                                                n \\ \kappa
                                                              \end{array}\right) \chi_{\infty}^{\kappa} \wedge\omega_{SF}^{n-\kappa} }-
\psi^{-}_{\epsilon} - \epsilon \varphi_{\infty} +\epsilon\log|S|^2_h\right) (t, z_{max,t})\\
\\
&\leq&- \psi^{-}_{\epsilon}(t,
z_{max,t})+\log \left( 1+ ( \AAA+1) \epsilon \right)+\epsilon.
\end{eqnarray*}
Applying the maximum principle again, we have
\begin{equation}
\psi^{-}_{\epsilon}\leq (\AAA+2)\epsilon+O(e^{-t}) \leq (\AAA +3) \epsilon),
\end{equation}
if we choose $\epsilon$ sufficiently small in the beginning and then $t$ sufficiently large.

\item If there exists $t_0\geq T_1$ such that $\max_{z\in
X}\psi^{-}_{\epsilon}(t_0, z)=\psi^{-}_{\epsilon}(t_0, z_0)<0$ for
some $z_0\in X$. Assume $t_1$ is the first time when $\max_{z\in
X, t\leq t_1}\psi^{-}_{\epsilon}(t, z)=\psi^{-}_{\epsilon}(t_1,
z_1) \geq ( \AAA + 3) \epsilon$. Then $z_1\in X\setminus
B_{r_{\epsilon}}(D)$ and applying the maximum
principle we have
\begin{eqnarray*}
\psi^{-}_{\epsilon}(t_1, z_1)&\leq&\left( \log\frac{e^{(n-\kappa)t}( (1+\epsilon\mu - e^{-t}) \chi_{\infty} + e^{-t}\omega_{SF,\epsilon})^n}
{\left(\begin{array}{c}
 n \\ \kappa
  \end{array}\right) \chi_{\infty}^{\kappa}\wedge\omega_{SF}^{n-\kappa}}-\psi^{-}_{\epsilon} - \epsilon \varphi_{\infty} +\epsilon\log|S|^2_h\right)
(t_1,z_1)\\
&&\\
&\leq&\log(1+ (\AAA + 1) \epsilon)+\epsilon <  (\AAA + 2 )\epsilon.
\end{eqnarray*}
which contradicts  the assumption that $\psi^{-}_{\epsilon}(t_1,
z_1)\geq (\AAA+3)\epsilon$.
Hence we have
\begin{equation*}
\psi^{-}_{\epsilon}\leq (\AAA+ 3)\epsilon.
\end{equation*}
By the same argument we have
\begin{equation*}
\psi^{+}_{\epsilon}\geq -(\AAA+ 3)\epsilon.
\end{equation*}

\end{enumerate}

\noindent This completes the proof.
\qed

\end{proof}

\medskip

\begin{proposition}\label{uniconv}
On any compact set $K$ of $X\backslash D$, we have \begin{equation} \lim_{t\rightarrow \infty}  ||\varphi(t, \cdot) - \varphi_{\infty}(\cdot) ||_{C^0(K)} = 0  .\end{equation}
\end{proposition}

\begin{proof}
By Proposition \ref{conv}, we have for $t>T_{\epsilon}$
\begin{equation*}
\varphi_{\infty}(t, z)+\epsilon\log|S|^2_h(t, z)-3\epsilon \leq
\varphi(t, z)\leq
                \varphi_{\infty}(t, z)-\epsilon\log|S|^2_h(t, z)+ 3\epsilon.
\end{equation*}

\noindent Then the proposition is proved by letting $\epsilon\rightarrow 0$.

\qed
\end{proof}


\subsection{K\"ahler-Ricci flow and minimal model program}

The K\"ahler-Ricci flow on algebraic manifolds of positive Kodaira
dimension seems to
be closely related to the minimal model program in algebraic geometry.

For any nonsingular minimal model $X$ of positive Kodaira dimension, the
canonical line bundle $K_X$ is nef and so the K\"ahler-Ricci flow
(\ref{krflow2}) has long time existence \cite{TiZha}. The abundance
conjecture predicts that $K_X$ is semi-ample, hence the canonical ring
of $X$ is finitely generated. If we assume the abundance conjecture, the
K\"ahler-Ricci flow will converge to the unique canonical metric on the
canonical  model $X_{can}$ associated to $X$ for any initial K\"ahler
metric by Theorem A.

If $X$ is not minimal, the K\"ahler-Ricci flow (\ref{krflow2}) will
develop finite time singularities. Let $T_1$ be the first time such that
$e^{-t}[\omega_0] - (1-e^{-t} ) 2\pi c_1(X)$ fails to be a K\"ahler
class. Adopting arguments in \cite{TiZha}, one can show that there is a
unique limiting current $\displaystyle \omega_{T_1}(\cdot) = \lim_{ t
\rightarrow T_1^{-} } \omega(t, \cdot) \in e^{-T_1}[\omega_0] -
(1-e^{-T_1} ) 2\pi c_1(X)$ and it is smooth outside an analytic
subvariety of $X$. Furthermore, the local potential $\varphi_{T_1}$ of
$\omega_{T_1}$ is continuous. We conjecture that $X_1$, the metric
completion of $\omega_{T_1}$, is again an algebraic variety and $X_1$
can be obtained by certain standard algebraic procedure such as a
blow-down or flip. It is reasonable to expect that such a variety $X_1$
does not have too bad singularities. In particular, we expect that a
weak K\"ahler-Ricci flow can be defined on $X_1$.  Suppose this is true,
we hope that the above procedure can be repeated as long as the
canonical line bundle is not nef. We further conjecture that after
repeating the above process finitely many times, we obtain the metric
completions $X_1$, $X_2$, ... , $X_N$ such that $K_{X_N}$ is nef !
Consequently, $X_N$ is a minimal model of $X$.

It provides a new understanding of the minimal model program from an
analytic point of view. We believe that it is interesting to further
explore this connection between the minimal model program and the study
of the regularity and convergence problem of the K\"ahler-Ricci flow on
algebraic varieties.


\section{Adjunction formulas for energy functionals}

\subsection{Generalized constant scalar curvature K\"ahler
metrics}

In fact, the canonical metrics in Section 3 belong to a class of K\"ahler
metrics which generalize Calabi's extremal metrics. Let $Y$ be a
K\"ahler manifold of complex dimension $n$ together with a fixed
closed (1,1)-form $\theta$. Fix a K\"ahler class $[\omega]$,
denote by ${\cal K}_{[\omega]}$ the space of K\"ahler metrics
within the same K\"ahler class, that is, all K\"ahler metrics of
the form $\omega_\varphi=\omega+ \sqrt{-1}\partial \bar\partial
\varphi$.
One may consider the following equation:
\begin{equation} \label{gene-C-equ} \bar\partial V_\varphi = 0,
\end{equation} where $V_\varphi$ is defined by \begin{equation}
\label{gene-C-equ2} \omega_{\varphi}(V_\varphi, \cdot) = \bar
\partial (S(\omega_\varphi) - \tr_{\omega_\varphi}(\theta)).
\end{equation}
Clearly, when $\theta=0$, (\ref{gene-C-equ}) is exactly the
equation for Calabi's extremal metrics. For this reason, we call a
solution of (\ref{gene-C-equ}) a generalized extremal metric. If
$Y$ does not admit any nontrivial holomorphic vector fields, then
any generalized extremal metric $\omega_\varphi$ satisfies
\begin{equation}
S(\omega_\varphi) - \tr_{\omega_\varphi}(\theta) = \mu,
\end{equation}
where $\mu$ is the constant given by
$$\mu=\frac{(2\pi c_1(Y) - [\theta])\cdot
[\omega]^{n-1}}{[\omega]^n}.$$ Moreover, if $2\pi c_1(Y)-
[\theta]=\lambda [\omega]$, then any such a metric satisfies
$$\ric(\omega_\varphi) = \lambda \omega_\varphi +\theta,$$
that is, $\omega_\varphi$ is a generalized K\"ahler-Einstein
metric. This can be proved by an easy application of the Hodge
theory. More interestingly, if we take $\theta$ to be the
pull-back of $\omega_{WP}$ by $f: X_{can}^{\circ}\rightarrow {\mathcal
M}_{CY}$, then we get back those generalized K\"ahler-Einstein
metrics which arise from limits of the K\"ahler-Ricci flow.

Let $f: X\rightarrow \Sigma$ be a K\"ahler surface admitting a
non-singular holomorphic fibration over a Riemann surface $\Sigma$ of genus greater than one, with fibres of
genus at least $2$. Let $V$ be
the vertical tangent bundle of $X$ and
$[\omega_t]=-f^*c_1(\Sigma)-tc_1(V)$.

Let $\chi$ be a K\"ahler form in $-c_1(\Sigma)$ and
$\omega_0 \in -c_1(V)$. Then $\omega_0=\omega_H\oplus \theta \chi$,
where $\omega_H$ is the hyperbolic K\"ahler form on each fiber and
$\theta$ is a smooth function on $X$. We then set
$$\omega_t=\chi+t\omega_0.$$
The following theorem is proved by Fine in \cite{Fi}.
\begin{theorem}
For $t>0$ sufficiently small, there exists a constant scalar
curvature K\"ahler metric in $[\omega_t]$. Furthermore, such a
family of constant scalar curvature K\"ahler metrics converge to a K\"ahler metric $\chi_{\infty}$ on
$\Sigma$ defined by

\begin{equation}
S(\chi_{\infty})-\tr_{\chi_{\infty}}(\theta)=const.
\end{equation}
where $\theta$ is the Weil-Petersson metric pulled back from the
moduli spaces of the fibre curves.

\end{theorem}


\subsection{Asymptotics of the Mabuchi energy by the large
K\"ahler structure limits}

Let $X$ be an $n$-dimensional compact K\"ahler manifold and $\omega$ a K\"ahler form. The Mabuchi energy functional $\mathcal{K}_{\omega}(\cdot)$ is
defined on $PSH(X, \omega)$ as follows

\begin{equation}\small
\mathcal{K}_{\omega}(\varphi)= \int_X
\log\frac{\omega_{\varphi}^n}{\omega^n}\omega_{\varphi}^n-\sum_{j=0}^{n-1}\int_X
\varphi \ric(\omega)\wedge \omega^{j}\wedge
\omega_{\varphi}^{n-j-1}+ \frac{n \mu}{n+1} \sum_{j=0}^{n}\int_X
\varphi \omega^{j}\wedge \omega_{\varphi}^{n-j},
\end{equation}
where $\omega_{\varphi} = \omega + \ddbar \varphi$ and $\mu=\frac{ 2\pi c_1(X) \cdot
[\omega]^{n-1}}{[\omega]^n}.$

\begin{definition}Let $X$ be a compact K\"ahler manifold of complex
dimension $n$. Let $\omega$ be a K\"ahler metric and $\theta$ a
closed $(1,1)$-form on $X$. Then the generalized Mabuchi energy
functional $\mathcal{K}_{\omega, \theta}(\cdot)$ is defined by
\begin{equation}\small
\mathcal{K}_{\omega,\theta}(\varphi)= \int_X
\log\frac{\omega_{\varphi}^n}{\omega^n}\omega_{\varphi}^n-\sum_{j=0}^{n-1}\int_X
\varphi (\ric(\omega)-\theta) \wedge \omega^{j}\wedge
\omega_{\varphi}^{n-j-1}+ \frac{n \mu}{n+1} \sum_{j=0}^{n}\int_X
\varphi \omega^{j}\wedge \omega_{\varphi}^{n-j},
\end{equation}
where $\mu=\frac{ (2\pi c_1(X) - [\theta])\cdot
[\omega]^{n-1}}{[\omega]^n}.$

\end{definition}

The following proposition can be proved by straightforward calculation.
\begin{proposition}

\begin{equation}
\delta \mathcal{K}_{\omega, \theta}=-\int_X \delta \varphi
\left(S(\omegaphi)-\tr_{\omegaphi}(\theta)\right)\omegaphi^n.
\end{equation}
Therefore
\begin{equation}\label{kenergy2}
\mathcal{K}_{\omega, \theta} (\varphi)=-\int_0^1\int_X  \dot{\varphi_t}
\left(S(\omega_t )-\tr_{\omega_t }(\theta)\right)\omega_t^n dt
\end{equation}
where $ \{ \varphi_t \}_{t\in[0,1]}$ is a smooth path in $PSH(X, \omega)$ with $\varphi_0=0$ and $\varphi_1= \varphi$, and $\omega_t = \omega + \ddbar \varphi_t$. The formula \ref{kenergy2} is independent of the choice of the path
$\varphi_t$.

\end{proposition}

Let $X$ be an $n$-dimensional compact K\"ahler manifold with semi-ample canonical line bundle. Suppose  $0< \kod(X) = \kappa < n$ and $X_{can} = X_{can}^{\circ}$, i.e., $X_{can}$ is nonsingular and the algebraic fibration $f: X \rightarrow X_{can}$ has no singular fibre.

Fix $\chi\in -c_1(X)$ as in Section 3.1 and $\omega_0$ an arbitrary
K\"ahler form with $\int_{X_s} \omega_{0,s}^{n-\kappa} = 1$, where $X_s = f^{-1}(s)$ and $\omega_{0, s} = \left. \omega_0 \right|_{X_s} $. Let $\omega_t=\chi+t\omega_0$ and
$\omega_{\varphi}=\omega_t+\ddbar\varphi$. Let $\omega_{SF}$ be
the semi-flat form in $[\omega_0]$. Then
$\varphi$ can always be decomposed as $$\varphi=\bar{\varphi}+t\psi$$ where
$\bar{\varphi}=\int_{X_y} \varphi \omega_{SF}^{n-k}$ is the
push-forward of $\varphi$ with respect to the flat metric on the
fibres.

\begin{theorem}\label{kenergy1}
Along the above class deformation of the K\"ahler
class on $X$,

\begin{equation}
\mathcal{K}_{\omega_t}(\varphi)={\tiny \left(%
\begin{array}{c}
  n \\
  \kappa \\
\end{array}%
\right)}t^{n-\kappa}\left(
\mathcal{K}_{\chi,\omega_{WP}}(\bvarphi)+ +
\mathcal{L}_{\chi,\chiphibar, \omega_0} (\psi) \right)+
O(t^{n-\kappa+1})
\end{equation}
and
\begin{eqnarray}\small
\mathcal{L}_{\chi,\chiphibar,\omega_0}(\psi)&=&\int_{s\in X_{can}} \left(\int_{X_s}\log
\frac{\omega_{\psi, s}^{n-\kappa}}{\omega_{0,s}^{n-\kappa}}\omega_{\psi, s}^{n-\kappa}\right)
\chiphibar^{\kappa}\\ \nonumber
&&- \sum_{j=0}^{n-\kappa-1}\sum_{i=0}^{\kappa} A_{i,j}
\int_{s\in X_{can}}\left( \int_{X_s} \psi \ric(\omega_{0, s})\wedge
\omega_{0,s}^{j}\wedge\omega_{\psi, s}^{n-\kappa-1-j} \right)
\chi^{i}\wedge \chiphibar^{\kappa-i},
\end{eqnarray}
where   $\small A_{i,j}= \tiny \left(%
\begin{array}{c}
  n \\
  \kappa \\
\end{array}%
\right) ^{-1} \left(%
\begin{array}{c}
  i+j \\
  i \\
\end{array}%
\right) \left(%
\begin{array}{c}
  n-1-i-j \\
  \kappa-i \\
\end{array}%
\right)$ and $\mathcal{K}_{\chi, \omega_{WP}}(\cdot)$ is the generalized Mabuchi energy on $X_{can}$.

In particular, when $\chi=\chiphibar$,
$$\mathcal{L}_{\chi,\chiphibar, \omega_0}(\psi)
=\int_{ s \in X_{can}} \mathcal{K}_{\omega_{0,s}}(\psi) \chi^{\kappa},$$
where $\mathcal{K}_{\chi}(\cdot)$ is the Mabuchi energy on
$X_{can}$, $\mathcal{L}_{\chi,\chiphibar\omega_0} (\cdot)$ is
defined by

\end{theorem}

\begin{proof} The proof boils down to direct computation. First, calculate \begin{equation} \small \int_X \log \frac{\omegaphi^n}{\omega^n} \omegaphi^n = t^{n-k}\left(%
\begin{array}{c}
  n \\
  \kappa \\
\end{array}%
\right) \left( \int_X \log
\frac{\chiphibar^{\kappa}}{\chi^{\kappa}}\chiphibar^{\kappa}+\int_{X_{can}}
\left(\int_{X_y}
\log\frac{\omegapsi^{n-\kappa}}{\omega^{n-\kappa}}\omegapsi^{n-\kappa}\right)\chiphibar^{\kappa}
+O(t)\right).
\end{equation}

\noindent Also
\begin{eqnarray*} \small
&&-\sum_{j=0}^{n-1}\int_X \varphi \ric(\omega)\wedge
\omega^{j}\wedge
\omega_{\varphi}^{n-j-1}\\
&=& t^{n-\kappa}\sum_{j=1}^{\kappa-1}
\left(%
\begin{array}{c}
  n \\
  \kappa \\
\end{array}%
\right)
 \int_{X_{can}} \bvarphi \left(
-\ric(\chi)+\omega_{WP}\right)\wedge \chi^j\wedge
\chiphibar^{\kappa-j-1}  \\
&&  -t^{n-\kappa}
\sum_{j=0}^{n-\kappa-1}\sum_{i=j}^{\kappa+j}A_{i,j} \int_{ s\in X_{can}}
\left( \int_{X_s} \ric(\omega_{0, s})\wedge
\omega_{0,s}^{j}\wedge\omega_{\psi, s}^{n-\kappa-1-j} \right)
\chi^{i}\wedge \chiphibar^{\kappa-i} +O(t^{n-\kappa+1})
\end{eqnarray*}
and
\begin{equation}
\sum_{j=0}^{n}\int_X \varphi \omega_t^j\wedge \omegaphi^{n-j}
=\left(%
\begin{array}{c}
  n+1 \\
  \kappa+1 \\
\end{array}%
\right) t^{n-\kappa} \sum_{j=0}^{\kappa} \int_{X_{can}} \bvarphi
\chi^j\wedge \chiphibar^{\kappa-j}+O(t^{n-\kappa+1}).
\end{equation}

The theorem follows from straightforward calculation by combining the above formulas.
\begin{eqnarray*} \small
&&\mathcal{K}_{\omega_t}(\varphi)\\
&=& {\tiny \left(%
\begin{array}{c}
  n \\
  \kappa \\
\end{array}%
\right) }t^{n-\kappa} \left(\int_{X_{can}} \log
\frac{\chiphibar^{\kappa}}{\chi^{\kappa}}\chiphibar^{\kappa}
-\int_{X_{can}}\bvarphi (\ric(\chi)-\omega_{WP})\wedge
\chi^{j}\wedge \chiphibar^{\kappa-j}
+\frac{\kappa\bar{\mu}}{\kappa+1}\int_{X_{can}}\bvarphi
\chi^{j}\wedge \chiphibar^{\kappa-j}\right)\\
&&+{\tiny \left(%
\begin{array}{c}
  n \\
  \kappa \\
\end{array}%
\right)} t^{n-\kappa} \int_{ s\in X_{can}} \left(\int_{X_s}\log
\frac{\omega_{\psi, s}^{n-\kappa}}{\omega_{0,s}^{n-\kappa}}\omega_{\psi, s}^{n-\kappa}\right)
\chiphibar^{\kappa}\\
&&- t ^{n -\kappa} \sum_{j=0}^{n-\kappa-1}\sum_{i=j}^{m+j} \int_{s \in X_{can}}
\left( \int_{X_s} \ric(\omega_{0, s})\wedge
\omega_{0,s}^{j}\wedge\omega_{\psi, s}^{n-\kappa-1-j} \right)
\chi^{i-j}\wedge \chiphibar^{m-i+j} \\
&&+O(t^{n-\kappa+1}),
\end{eqnarray*}
where $ \bar{\mu} = \frac{ ( 2\pi c_1(X_{can}) - [\omega_{WP}] ) \cdot [\chi]^{\kappa -1} } { [\chi]^{\kappa}}$.

\qed
\end{proof}

We also investigate the asymptotics of the Mabuchi energy in the
case of a fibred space studied by Fine in \cite{Fi}.

Let $f: X\rightarrow \Sigma$ be a K\"ahler surface admitting a
non-singular holomorphic fibration over $\Sigma$, with fibres of
genus at least $2$. We also assume $c_1(\Sigma) < 0$. Let $V$ be
the vertical tangent bundle of $X$ and
$[\omega_t]=-f^*c_1(\Sigma)-tc_1(V)$.

Let $\chi$,
$\omega_0 \in-c_1(V)$ and $\omega_t$ be defined as in Section 6.1.
We consider the asymptotics of the Mabuchi energy $\mathcal{K}_{\omega_t}(\cdot)$ as $t$ tends to
$0$.

\begin{theorem}\label{kenergy3}
Let $\omega_0\in -c_1(V)$ be a closed $(1,1)$ such that its
restriction on each fibre is a hyperbolic metric. Let
$\omega_t=\chi+t\omega_0$ and $\omegaphi=\omega_t+\ddbar \varphi$
be a metric deformation, where $\varphi\in C^{\infty}(\Sigma)$.
Then we have

\begin{equation}
\mathcal{K}_{\omega_t}(\varphi)= 2t \mathcal{K}_{\chi,
\theta}(\varphi)+ O(t^{2}).
\end{equation}

\end{theorem}

Theorem \ref{kenergy1} and Theorem  \ref{kenergy3} can be considered as an adjunction type formula for the Mabuchi energy on an algebraic fibre space.

\bigskip
\bigskip

\noindent {\bf Acknowledgments } The first named author would like
to H. Fang, D.H. Phong, J. Sturm, V. Tosatti, B. Weinkove and Z. Zhou for  enlightening discussions. Both authors thank
MSRI for its hospitality during the program of geometric evolution equations and related topics where part of the work was carried out.
The second named author would also like to thank the Clay Math. Inst. for financial support during his visit to MSRI.


\end{document}